\documentclass[reqno]{amsart}

\usepackage{fixltx2e}
\usepackage[T1]{fontenc}
\usepackage{amsmath}
\usepackage{amsfonts}
\usepackage{amssymb}
\usepackage{calrsfs}
\usepackage[all]{xy}
\usepackage{mathbbol}
\usepackage{mathabx}
\usepackage{mathdots}
\usepackage{enumerate}
\usepackage[pdftex]{xcolor}
\usepackage[pdftex]{graphicx}

\theoremstyle{plain}
\newtheorem{theorem}[subsection]{Theorem}
\newtheorem{lemma}[subsection]{Lemma}
\newtheorem{proposition}[subsection]{Proposition}
\newtheorem{corollary}[subsection]{Corollary}

\theoremstyle{definition}
\newtheorem{definition}[subsection]{Definition}
\newtheorem{remark}[subsection]{Remark}
\newtheorem{example}[subsection]{Example}

\newcommand{\defn}{\textbf}

\newcommand{\comp}{\raisebox{0.2mm}{\ensuremath{\scriptstyle{\circ}}}}
\newcommand{\del}{\partial}
\newcommand{\To}{\Rightarrow}

\newcommand{\DefEq}{\coloneq} 
\newcommand{\EqDef}{\eqcolon} 
\newcommand{\links}{\lgroup}
\newcommand{\rechts}{\rgroup}

\newcommand{\Centr}{\ensuremath{\mathrm{Centr}}}

\renewcommand{\H}{\ensuremath{\mathrm{H}}}
\renewcommand{\hom}{\ensuremath{\mathrm{Hom}}}

\newcommand{\Ext}{\ensuremath{\mathrm{Ext}}}

\newcommand{\Nat}{\ensuremath{\mathrm{Nat}}}
\newcommand{\Ob}[1]{\ensuremath{|#1|}}
\newcommand{\R}{\ensuremath{\mathcal{R}}}
\newcommand{\op}{\ensuremath{\mathrm{op}}}

\newcommand{\A}{\ensuremath{\mathcal{A}}}

\newcommand{\C}{\ensuremath{\mathcal{C}}}
\newcommand{\F}{\ensuremath{\mathcal{F}}}
\newcommand{\Ell}{\ensuremath{\mathsterling}}

\newcommand{\X}{\ensuremath{\mathcal{X}}}
\newcommand{\Y}{\ensuremath{\mathcal{Y}}}
\newcommand{\T}{\ensuremath{\mathcal{T}}}

\newcommand{\Ab}{\ensuremath{\mathsf{Ab}}}
\newcommand{\ACat}{\ensuremath{\mathsf{ACat}}}

\newcommand{\CatExt}{\ensuremath{\mathsf{Ext}}}
\newcommand{\CExt}{\ensuremath{\mathsf{CExt}}}
\newcommand{\Gp}{\ensuremath{\mathsf{Gp}}}
\newcommand{\Lie}{\ensuremath{\mathsf{Lie}}}
\newcommand{\Mod}{\ensuremath{\mathsf{Mod}}}
\newcommand{\NatCat}{\ensuremath{\mathsf{Nat}}}
\newcommand{\SACat}{\ensuremath{\mathsf{SACat}}}

\newcommand{\ab}{\ensuremath{\mathrm{ab}}}
\newcommand{\CC}{\ensuremath{\mathrm{C}}}
\newcommand{\D}{\ensuremath{\mathrm{D}}}
\renewcommand{\L}{\ensuremath{\mathrm{L}}}

\newcommand{\XX}{\ensuremath{{X}}}
\newcommand{\YY}{\ensuremath{\mathsf{Y}}}
\newcommand{\ZZ}{\ensuremath{\mathbb{Z}}}

\newcommand{\from}{\colon}				
\newcommand{\Lm}{\(}	
\newcommand{\Rm}{\)}	
\newcommand{\LM}{\hskip +.1em\(}	
\newcommand{\RM}{\) \hskip +.1em}	
\newcommand{\ZNr}[1][]{\mathbb{Z}^{#1}}		
\newcommand{\IdMap}[1]{1_{#1}}		
\newcommand{\Img}[1]{\text{\rm Im}(#1)}	
\newcommand{\HomFnctr}[3]{\text{\rm Hom}_{#1}(#2,#3)} 
\newcommand{\NatTrafos}[2]{\mathrm{Nat}(#1,#2)}	
\newcommand{\Ker}[1]{\mathrm{Ker}(#1)}			
\newcommand{\CoKer}[1]{\text{\rm Coker}(#1)}	
\newcommand{\Prdct}[2]{#1\hskip -.1em\times\hskip -.1em #2} 
\newcommand{\CPrdct}[2]{#1 + #2}			
\newcommand{\AbCore}[1]{\mathsf{Ab}(#1)} 
\newcommand{\AdjUnit}[1][]{\eta_{#1}}					

\newcommand{\AdjCUnit}[1][]{\varepsilon_{#1}}	
\newcommand{\Sets}{\mathsf{Set}}	 
\newcommand{\Grps}{\mathsf{Gp}}		
\newcommand{\AbGrps}{\mathsf{Ab}}	
\newcommand{\LModules}[1]{{}_{#1}\mathsf{Mod}} 
\newcommand{\ChnCat}{\mathsf{ch}} 
\newcommand{\HmlgyKer}[2]{\H^{\mathrm{k}}_{#1}#2}							
\newcommand{\HmlgyCoKer}[2]{\H^{\mathrm{c}}_{#1}#2}					
\newcommand{\CoHmlgyCoKer}[2]{\H_{\mathrm{c}}^{#1}#2}					
\newcommand{\ThreeCat}{\langle 3\rangle} 
\newcommand{\TrsnCat}[1][]{\T_{#1}}	
\newcommand{\ExtFnctr}[4][]{\text{\rm Ext}_{#1}^{#2}(#3,#4)}	
\newcommand{\LDrvd}[3]{\L_{#2}#1(#3)}													
\newdir{>>}{{}*!/3.5pt/:(1,-.2)@^{>}*!/3.5pt/:(1,+.2)@_{>}*!/7pt/:(1,-.2)@^{>}*!/7pt/:(1,+.2)@_{>}}
\newdir{ >>}{{}*!/8pt/@{|}*!/3.5pt/:(1,-.2)@^{>}*!/3.5pt/:(1,+.2)@_{>}}
\newdir{ |>}{{}*!/-3.5pt/@{|}*!/-8pt/:(1,-.2)@^{>}*!/-8pt/:(1,+.2)@_{>}}
\newdir{ >}{{}*!/-8pt/@{>}}
\newdir{>}{{}*:(1,-.2)@^{>}*:(1,+.2)@_{>}}
\newdir{<}{{}*:(1,+.2)@^{<}*:(1,-.2)@_{<}}

\def\pullback{
 \ar@{-}[]+R+<6pt,-1pt>;[]+RD+<6pt,-6pt>%
 \ar@{-}[]+D+<1pt,-6pt>;[]+RD+<6pt,-6pt>}

\def\pushout{%
 \ar@{-}[]+L+<-6pt,1pt>;[]+LU+<-6pt,6pt>%
 \ar@{-}[]+U+<-1pt,6pt>;[]+LU+<-6pt,6pt>}

\def\splitpullback{%
 \ar@{-}[]+R+<6pt,-.51ex>;[]+RD+<6pt,-6pt>%
 \ar@{-}[]+D+<.51ex,-6pt>;[]+RD+<6pt,-6pt>}

\def\skewpullback{%
 \ar@{-}[]+LD+<-6pt,-6pt>;[]+LDD+<-6pt,-15.5pt>%
 \ar@{-}[]+D+<-1pt,-6pt>;[]+LDD+<-6pt,-15.5pt>}

\begin{document}
\author{George Peschke}
\author{Tim Van~der Linden}

\email{george.peschke@ualberta.ca}
\email{tim.vanderlinden@uclouvain.be}

\address{Dept.~of Mathematical and Statistical Sciences, University of Alberta, Edmonton AB, Canada T6T 2G1}
\address{Institut de Recherche en Math\'ematique et Physique, Universit\'e catholique de Louvain, chemin du cyclotron~2 bte~L7.01.02, 1348 Louvain-la-Neuve, Belgium}

\thanks{We thank the Banff International Research Station for Mathematical Innovation and Discovery (BIRS), Stellenbosch University, and the Institut de Recherche en Math\'ematique et Physique (IRMP) at the Universit\'e catholique de Louvain for their kind hospitality.}
\thanks{The second author is a Research Associate of the Fonds de la Recherche Scientifique--FNRS. He began studying some of the preliminary material on half-exactness of functors during a stay at York University in 2007 supported by Walter Tholen's NSERC research grant.}

\title[The Yoneda isomorphism commutes with homology]{The Yoneda isomorphism \\ commutes with homology}

\begin{abstract}
We show that, for a right exact functor from an abelian category to abelian groups, Yoneda's isomorphism commutes with homology and, hence, with functor derivation. Then we extend this result to semiabelian domains. An interpretation in terms of satellites and higher central extensions follows. As an application, we develop semiabelian (higher) torsion theories and the associated theory of (higher) universal (central) extensions.
\end{abstract}

\date{\today}

\keywords{Yoneda's lemma, semiabelian category, derived functor, extension}

\subjclass[2010]{Primary 18E99, 18G50; Secondary 18E25, 18E40, 20J05}

\maketitle

\section*{Introduction}
Yoneda's lemma establishes for any \LM \Sets\Rm-valued functor \LM {F\from \X\to \Sets}\RM and any object $U$ of $\X$ the bijection
$$
\NatTrafos{ \HomFnctr{\X}{ U }{ - } }{F} \longrightarrow FU,\quad \tau\mapsto \tau(\IdMap{U}).
$$
We refer to this bijection as the \defn{Yoneda isomorphism}~\cite[pp.~59ff.]{MacLane}, and note that it follows without any further assumptions from the structural axioms of categories. It therefore belongs to the very foundations of the theory of categories.

If $\X$ is enriched in the category of abelian groups \LM \AbGrps\Rm, and $F$ is additive and \LM \AbGrps\Rm-valued, then the Yoneda isomorphism is a morphism of abelian groups. This may be verified directly, or seen as a special case of results in~\cite{Kelly:Enriched}.

Extending work of Mac Lane~\cite{MacLane:Duality}, as a critical refinement of \LM \AbGrps\Rm-enriched categories, Buchsbaum~\cite{Buchsbaum:ExactCats} and Grothendieck~\cite{Tohoku} introduced abelian categories to provide a purely axiomatic foundation for homological algebra. The key result of the present paper establishes a \emph{homological} Yoneda isomorphism, Lemma~\ref{Lemma-Yoneda-Commutes-Homology}: For a chain complex~$C$ in an abelian category~$\A$, and a right exact functor $T$ from $\A$ to the category of abelian groups,
$$
\NatTrafos{\H^n \HomFnctr{\A}{ C }{ - } }{T} \xrightarrow{\quad \cong\quad} \H_nTC.
$$
This result belongs to the very foundations of abelian categories as it follows without any further assumptions from their structural axioms, and this is visible in its proof: Yoneda's lemma turns the claim into a tautology upon contemplating the self-dual nature of the concept of `chain complex' and what this means to `homology'.

The homological Yoneda isomorphism immediately implies results about the left derived functors of $T$ which go back as far as Yoneda~\cite{Yoneda-Exact-Sequences} and Hilton--Rees~\cite{Hilton-Rees}, see also~\cite{Hilton-Stammbach}: in particular, the existence of natural isomorphisms
$$
\NatTrafos{\ExtFnctr{n}{M}{-}}{T} \xrightarrow{\cong} \LDrvd{T}{n}{M}.
$$
Moreover, the homological Yoneda isomorphism is compatible with the connecting morphism in the long exact homology sequences associated to a short exact sequence of chain complexes; see Theorem~\ref{Theorem-Yoneda-Connecting-Map}.

The much younger concept of semiabelian categories was introduced in~\cite{Janelidze-Marki-Tholen}. It constitutes a synthesis between previous efforts toward an axiomatic foundation for homological algebra in the absence of enrichment in abelian groups---see, for instance,~\cite{Huq, Gerstenhaber, Orzech}---and recent discoveries in Categorical Algebra, such as the interplay between Barr exactness~\cite{Barr} and Bourn protomodularity~\cite{Bourn1991, Bourn2001}. Examples of semiabelian categories are the categories of groups, loops, rings, associative algebras, Lie algebras, crossed modules, and many others.

We show that the homological Yoneda lemma is also valid for (sequentially) right exact functors from a semiabelian category $\X$ to the category of abelian groups; see~\ref{Lemma-Yoneda-Iso-Commutes-Homology-Semiab}; see~\ref{def:Sequentially(R)EFunctor} for the definition of `sequentially right exact functor'. We view this as additional evidence of the ideal adaptedness of semiabelian categories to their intended purpose.

As an application of the homological Yoneda lemmas~\ref{Lemma-Yoneda-Commutes-Homology} and~\ref{Lemma-Yoneda-Iso-Commutes-Homology-Semiab} we develop a general theory of higher universal central extensions in a suitable semiabelian variety $\X$: A reflector $T$ from $\X$ to a subcategory $\R$ of its abelian core determines a family of object classes \LM \TrsnCat[0]\supseteq \TrsnCat[1]\supseteq \cdots \supseteq \TrsnCat[n]\supseteq \cdots\RM in~$\X$ as follows: An object~$X$ belongs to $\TrsnCat[n]$ if the left derived functors \LM \LDrvd{T}{i}{X}\RM vanish for \LM 0\leq i\leq n\Rm. We identify~\LM \TrsnCat[0]\RM as the torsion class of a (non-abelian) torsion theory. Accordingly, for \LM n\geq 1\Rm, the \LM \TrsnCat[n]\RM could be called \emph{higher torsion classes}. Via the homological Yoneda isomorphism~\ref{Lemma-Yoneda-Iso-Commutes-Homology-Semiab} we prove that $X$ belongs to \LM \TrsnCat[n]\RM if and only if \LM \ExtFnctr{0}{X}{-}=\cdots =\ExtFnctr{n}{X}{-}=0\Rm; see~\ref{theorem:XinT_i<->Ext^i(X,LY)=0}.

Now consider an object $X$ in \LM \TrsnCat[n-1]\Rm. The universal coefficient theorem~\ref{Universal coefficient theorem} establishes a natural equivalence of the functors
\[
\left(\HomFnctr{\R}{\LDrvd{T}{n}{X}}{-}\overset{\cong}{\implies} \ExtFnctr{n}{X}{-}|_\R\right)\from \R \longrightarrow \AbGrps.
\]
When $\R$ is an abelian subvariety of a semi-abelian variety $\X$ satisfying the so-called \emph{Smith is Huq} condition~\ref{SH}, recent work in~\cite{RVdL2} tells us that for $M$ in~$\R$, the group \LM \ExtFnctr{n}{X}{M}\RM is (naturally) isomorphic to the group of equivalence classes of central $n$-extensions under $M$ and over $X$. Hence Yoneda's lemma yields the existence of a universal central $n$-extension under \LM \LDrvd{T}{n}{X}\RM and over $X$; see~\ref{thm:UCEsExst}.

Thus we lift results from~\cite{Peschke-UE} about module categories to a semiabelian setting, and known results on one-fold central extensions~\cite{CVdL} to higher degrees. The outcome is new, even in the categories of groups and Lie algebras. An explicit construction of higher universal central extensions is given in~\ref{Construction}. Finally, we should point out that the definition of universality of a higher central extension is not a direct analogue of the classical one dimensional definition; see Section \ref{sec:UCEs} for details.

\subsection*{Organisation of the Paper}
We have endeavoured to make this paper as accessible as possible. For example, Sections~\ref{sec:Yoneda-Commutes-Homology-SemiAb} and~\ref{sec:UCEs} are written so that they have immediate and meaningful interpretations from both (1) the Barr--Beck view of cotriple derived functors~\cite{Barr-Beck} and (2) the view of Quillen's simplicial model categories via~\cite[II, Theorem 4]{Quillen} and~\cite{VdLinden:Simp}.

In Section~\ref{sec:Yoneda-Iso-Commutes-Homology-AbCase} we prove the homological Yoneda lemma and infer immediate consequences. In Section~\ref{sec:SACats} we collect background on semiabelian categories. In Section~\ref{sec:SequentiallyREFunctors} we establish a deeper link between the (large) $2$-category \LM \ACat\RM of abelian categories and right exact functors, and the (large) $2$-category \LM \SACat\RM of semiabelian categories and sequentially right exact functors: the `abelian core' is a pseudoreflector from the latter to the former; see Theorem~\ref{thm:ACat-Reflective-In-SACat}. In Section~\ref{sec:Yoneda-Commutes-Homology-SemiAb} we extend the results of Section~\ref{sec:Yoneda-Iso-Commutes-Homology-AbCase} to right exact functors from a semiabelian category to the category of abelian groups. In Section~\ref{sec:HigherExtensions} we outline the theory of higher central extensions in semiabelian categories. This enables us to (1) interpret the results of Section~\ref{sec:Yoneda-Commutes-Homology-SemiAb} in terms of satellites and (2) develop a theory of higher universal central extensions in Section~\ref{sec:UCEs}. In Appendix~\ref{sec:HomologyBackground} we present facts about chain complexes in abelian categories and their homology. While this material is elementary, the view offered here helps understand why Yoneda's isomorphism commutes with homology.

\subsection*{Acknowledgements} We are grateful for the referee's comments which have\linebreak helped make the paper clearer and more generally accessible. It is a pleasure to thank George Janelidze for adding his insight; see the passage `Added in proof' at the very end.

\section{Yoneda's Isomorphism Commutes with Homology: Abelian Case}
\label{sec:Yoneda-Iso-Commutes-Homology-AbCase}%

In this section we prove the homological Yoneda lemma~\ref{Lemma-Yoneda-Commutes-Homology}, and show that it is compatible with the connecting morphism in the long exact homology sequence of a short exact sequence of chain complexes: Theorem~\ref{Theorem-Yoneda-Connecting-Map}.

\begin{lemma}[Homological Yoneda lemma]%
\label{Lemma-Yoneda-Commutes-Homology}%
Let \LM T\from \A\to \AbGrps\RM be a right exact functor from an abelian category to the category of abelian groups, and let $C$ be a chain complex in~$\A$. If \LM n\in \ZNr\Rm, there is an isomorphism, natural in $C$ and in $T$,
$$
\YY_n(C)\from \NatTrafos{\H^n\hom_{\A}(C,-)}{T} \overset{\cong}{\longrightarrow} \H_nTC.
$$
\end{lemma}
\begin{proof}
As in~\ref{subsec:Homology}, we use here both the kernel and the cokernel constructions of homology. Setting \LM \HomFnctr{}{\cdot}{-}\DefEq \HomFnctr{\A}{\cdot}{-}\Rm, we compute:
\[
\resizebox{\textwidth}{!}
{\mbox{$\begin{aligned}
\NatTrafos{\CoHmlgyCoKer{n}{\HomFnctr{}{C}{-}}}{T}
 & \cong \NatTrafos{\CoKer{\HomFnctr{}{ C_{n-1} }{ - } \to \Ker{d^{\ast}_{n+1}}}}{T} \\
 & \cong \Ker{\NatTrafos{ \Ker{d^{\ast}_{n+1}} }{T} \to \NatTrafos{ \HomFnctr{}{ C_{n-1} }{-} }{ T } } \\
 & \cong \Ker{ \NatTrafos{ \HomFnctr{}{ \CoKer{d_{n+1}} }{ - } }{T} \to \NatTrafos{ \HomFnctr{}{ C_{n-1} }{-} }{ T } } \\
 & \cong \Ker{ T\CoKer{d_{n+1}} \to TC_{n-1} } \\
 & \cong \Ker{ \CoKer{Td_{n+1}} \to TC_{n-1} } \\
 &= \HmlgyKer{n}{TC} \cong \HmlgyCoKer{n}{TC}
\end{aligned}$}}
\]
These isomorphisms exist for the following reasons (in order of appearance):
\begin{enumerate}[(1)]
\item the standard definition of cohomology of a cochain complex;
\item continuity of Nat in the first variable;
\item continuity of Hom in the first variable;
\item the ordinary Yoneda lemma applied twice;
\item right exactness of $T$;
\item the kernel definition of homology of a chain complex;
\item follows via the snake interchange argument in~\ref{subsec:Homology}. \qedhere
\end{enumerate}
\end{proof}

\begin{remark}
Lemma~\ref{Lemma-Yoneda-Commutes-Homology} may be viewed as a generalisation of Yoneda's lemma as follows: Given an object $M$ in $\A$, let \LM K(M,0)\RM be the chain complex with $M$ in position $0$, and zero objects everywhere else. Then
$$
\H^0 \HomFnctr{\A}{ K(M,0) }{ - }\cong \HomFnctr{\A}{ M }{ - } \quad\text{and}\quad \H_0TK(M,0)\cong TM.
$$
In this special case,~\ref{Lemma-Yoneda-Commutes-Homology} is the standard Yoneda lemma.
\end{remark}

\begin{theorem}\label{Theorem-Yoneda-Connecting-Map}
Let \LM T\from {\A\to \AbGrps}\RM be a right exact functor on an abelian category. Then every term-wise split short exact sequence of chain complexes in $\A$
\begin{equation}\label{Short-Exact-Sequence-of-Chain-Complexes}
\xymatrix{0 \ar[r] & A \ar@{{ |>}->}[r]^-{\alpha} & B \ar@{-{ >>}}[r]^-{\beta} & C \ar[r] & 0} \tag{$\star$}
\end{equation}
determines an isomorphism of long exact sequences of abelian groups:
\[
\resizebox{\textwidth}{!}
{\xymatrix@R=6ex@C=2em{
\NatTrafos{\H^{n+1}\HomFnctr{}{C}{-} }{ T } \ar@{}[rd]|-{(*)} \ar@<+10pt>[d]_{\cong}^{\YY_{n+1}(C)} \ar[r]^-{(\delta^{n+1})^\asterisk} &
 \NatTrafos{ \H^{n}\HomFnctr{}{A}{-} }{ T } \ar[r]^-{(\alpha^\asterisk)^\asterisk} \ar[d]_{\YY_n(A)}^{\cong} &
 \NatTrafos{ \H^{n}\HomFnctr{}{B}{-} }{ T } \ar[r]^-{(\beta^\asterisk)^\asterisk} \ar[d]_{\YY_n(B)}^{\cong} &
 \NatTrafos{ \H^{n}\HomFnctr{}{C}{-} }{ T } \ar@<-10pt>[d]_{\YY_n(A)}^{\cong} \\
\cdots \to \H_{n+1}{T}{C} \ar[r]_-{\partial_{n+1}} &
 \H_{n}{T}{A} \ar[r]_-{T\alpha} &
 \H_{n}{T}{B} \ar[r]_-{T\beta} &
 \H_{n}{T}{C}\to \cdots}}
\]
Moreover, this isomorphism is natural with respect to $T$ and morphisms of short exact sequences~\eqref{Short-Exact-Sequence-of-Chain-Complexes}.
\end{theorem}
\begin{proof}
To begin let us note that the half exact functors \LM \HomFnctr{}{\cdot}{-}\DefEq\HomFnctr{\A}{\cdot}{-}\RM and $T$ turn the given split exact sequence of chain complexes into split exact sequences of (co\nobreakdash-)chain complexes. In view of Theorem~\ref{Lemma-Yoneda-Commutes-Homology} it only remains to show that the rectangle $(*)$ on the left commutes. The argument is an extension of the proof of Theorem~\ref{Lemma-Yoneda-Commutes-Homology}: While, for every chain complex $U$ in $\A$, Yoneda's lemma turns the (cokernel) description of~\LM \CoHmlgyCoKer{n}{ \HomFnctr{}{ U }{ - } }\RM into the (kernel) description of \LM \HmlgyKer{n}{TU}\Rm, it turns the corresponding description of the connecting morphism \LM (\delta^{n+1}_{\mathrm{c}})^{\ast}\RM associated to the short exact sequence of cochain complexes
$$
\xymatrix@C=1.5em{
0 &
 \HomFnctr{}{ A }{ - } \ar[l] &&
 \HomFnctr{}{ B }{ - } \ar@{-{ >>}}[ll]_-{\alpha^*} &&
 \HomFnctr{}{ C }{ - } \ar@{{ |>}->}[ll]_-{\beta^*} &
 0 \ar[l]
}
$$
into the description of connecting morphism \LM \partial_{n+1}^{\mathrm{k}}\from \HmlgyKer{n+1}{TC}\to \HmlgyKer{n}{TA}\Rm; see~\ref{subsec:Homology}. Equivalently, since the covariant functor \LM \NatTrafos{\cdot}{T}\from (\Ab^\A)^\op\to \Ab\RM preserves limits, it sends the kernel description of \LM (\delta^{n+1}_{\mathrm{c}})^{\op}\RM to the kernel description of \LM \del_{n+1}^{\mathrm{k}}\Rm. Either way, we see that the rectangle in question commutes. Naturality of this morphism of long exact sequences follows from the naturality of the operations involved.
\end{proof}

As an immediate corollary we obtain the classical formula of Yoneda~\cite{Yoneda-Exact-Sequences} and Hilton--Rees~\cite{Hilton-Rees} for left derived functors, plus its compatibility with the connecting homomorphism.

\begin{theorem}\label{Theorem-Yoneda-Derived-Connecting-Map}
Let $\A$ be an abelian category with enough projectives and let
\begin{equation*}
\xymatrix@R=5ex@C=1.5em{
0 \ar[r] &
 K \ar@{{ |>}->}[rr]^-{k} &&
 X \ar@{-{ >>}}[rr]^-{f} &&
 Y \ar[r] &
 0}
\end{equation*}
be a short exact sequence in $\A$. Then a right exact functor \LM T\from{\A\to \AbGrps}\RM yields the isomorphism of long exact sequences of abelian groups below.
\[
\resizebox{\textwidth}{!}
{\xymatrix@R=6ex@C=2em{
\NatTrafos{\ExtFnctr{n+1}{Y}{-} }{ T } \ar@<+10pt>[d]_{\cong}^{\YY_{n+1}(Y)} \ar[r]^-{(\delta^{n+1})^\asterisk} &
 \NatTrafos{ \ExtFnctr{n}{K}{-} }{ T } \ar[r]^-{(k^\asterisk)^\asterisk} \ar[d]_{\YY_n(K)}^{\cong} &
 \NatTrafos{ \ExtFnctr{n}{X}{-} }{ T } \ar[r]^-{(f^\asterisk)^\asterisk} \ar[d]_{\YY_n(X)}^{\cong} &
 \NatTrafos{ \ExtFnctr{n}{Y}{-} }{ T } \ar@<-10pt>[d]_{\YY_n(Y)}^{\cong} \\
\cdots \to \LDrvd{T}{n+1}{Y} \ar[r]_-{\partial_{n+1}} &
 \LDrvd{T}{n}{K} \ar[r]_-{k_\asterisk} &
 \LDrvd{T}{n}{X} \ar[r]_-{f_\asterisk} &
 \LDrvd{T}{n}{Y}\to \cdots}}
\]
\end{theorem}
\begin{proof}
Cover the given short exact sequence by a short exact sequence of projective chain complexes, and apply Theorem~\ref{Theorem-Yoneda-Connecting-Map}.
\end{proof}

\begin{remark}[Enrichment in \LM \LModules{R}\Rm]
\label{rem:Yoneda-Iso-Commutes-Homology-RMod}%
Given a ring $R$, and a left exact functor \LM T\from \A\to \LModules{R}\Rm, we may equivalently regard $T$ as taking values in the category of (left $R$ and right $\ZNr$)-bimodules. Accordingly, the Yoneda isomorphism
$$
\NatTrafos{ \HomFnctr{\A}{ X }{ - } }{ T } \to TX
$$
is an isomorphism of $(R,\ZNr)$-bimodules. This bimodule compatibility passes through the proofs of the results in the current section.
\end{remark}

\section{Background on Semiabelian Categories}%
\label{sec:SACats}%

Extending work of Mac Lane \cite{MacLane:Duality}, abelian categories were designed in~\cite{Buchsbaum:ExactCats} and~\cite{Tohoku} as an axiomatic framework for classical homological algebra. From the 1960s on, several dissimilar seeming approaches aimed at generalizing this framework so as to include all varieties such as `groups', `Lie algebras', `loops', `non-unitary rings', etc.---see, for instance,~\cite{Huq, Gerstenhaber, Orzech}. In order to unify these early axiom systems, together with the results using them, the much younger concept of semiabelian categories was introduced in~\cite{Janelidze-Marki-Tholen}.
%

Such unification had become possible via recent developments in Categorical Algebra, in particular the concepts of Barr exactness~\cite{Barr} and Bourn protomodularity~\cite{Bourn1991, Bourn2001}. The result is a delicate balance between sacrificing enrichment in abelian groups, while preserving the validity of certain diagram lemmas whose use is essential in homological algebra. Examples include the Short Five Lemma, the \LM (3\times 3)\Rm-Lemma, and the Snake Lemma.

For a nice introductory survey to semiabelian categories we suggest~\cite{Borceux-Semiab, Bourn-Gran-CategoricalFoundations} and for a comprehensive treatment we recommend~\cite{Borceux-Bourn}. In the following we highlight some well-known facts which are essential to our purpose.

\subsection{The Abelian Core of a Semiabelian Category}
\label{subsec:AbCore}%
Whenever an object in a semiabelian category $\X$ admits an internal abelian group structure, this structure is necessarily unique. The collection of all abelian group objects in $\X$ determines a full subcategory \LM \AbCore{\X}\RM of $\X$, called the \defn{abelian core}. It is an abelian subcategory of~$\X$ which is reflective in $\X$. Its unit `abelianisation' consists of regular epimorphisms \LM {\AdjUnit[X]\from X\to \ab_{\X}(X)}\Rm. Using the treatment of cooperating monomorphisms in~\cite[p.~28ff]{Borceux-Bourn}, abelianisation may be constructed via the pushout diagram below.
$$
\xymatrix@R=5ex@C=3em{
\CPrdct{X}{X} \ar@{-{ >>}}[r]^-{\left\links\begin{smallmatrix}\IdMap{X} & 0\\
0&\IdMap{X}\end{smallmatrix}\right\rechts} \ar@{-{ >>}}[d]_-{\left\links\begin{smallmatrix}\IdMap{X}& \IdMap{X}\end{smallmatrix}\right\rechts} &
 \Prdct{X}{X} \ar@{-{ >>}}[d] \\
X \ar@{-{ >>}}[r]_-{\AdjUnit[X]} &
 \ab_{\X}(X) \pushout
}
$$

\begin{example}
\label{exa:AbCore(SAVariety)}
\cite[p.~106]{Freyd} The abelian core of a semiabelian variety~$\X$ is the category \LM \LModules{R}\Rm, where $R$ is the endomorphism ring of the abelianisation of the free object on a single generator.
\end{example}

\begin{lemma}\label{Lemma-YonedaRestriction}
Let \LM T\from \X\to \R\RM be a reflector on a semiabelian category $\X$ onto a full and replete subcategory $\R$ of its abelian core, and let \LM F\from \R\to \AbGrps\RM be an arbitrary functor. Then the isomorphism below is natural in $X$ and $F$.
$$
FT(X) \cong \NatTrafos{ \HomFnctr{\X}{X}{-}|_{\R} }{F}
$$
\end{lemma}
\begin{proof}
We have the natural equivalence
\[
\vartheta=(\vartheta_X^{\ast})_{X}\from \HomFnctr{\X}{TX}{-}|_{\R}\To \HomFnctr{\X}{X}{-}|_{\R},
\]
induced by the reflection unit \LM \vartheta_X\from X\to TX\Rm. Therefore,
\begin{align*}
FTX \cong &\; \NatTrafos{\HomFnctr{\R}{TX}{-}}{ F } \\
 = &\; \NatTrafos{ \HomFnctr{\X}{TX}{-}|_{\R}}{F} \\
 \cong &\; \NatTrafos{ \HomFnctr{\X}{X}{-}|_{\R}}{F}.\qedhere
\end{align*}
\end{proof}

\begin{example}[Groups]
In Lemma~\ref{Lemma-YonedaRestriction}, take \LM \X\DefEq \Grps\RM and \LM F\DefEq \ab|_{\AbGrps}\Rm, the restriction of `abelianisation' on \LM \Grps\RM to abelian groups. Then, for any group~$X$,
\[
\ab(X)\cong\NatTrafos{ \HomFnctr{\AbGrps}{\ab(X)}{-} }{ \IdMap{\AbGrps} } \cong \NatTrafos{\HomFnctr{\Grps}{X}{-}|_{\AbGrps} }{ \ab|_{\Ab}}.
\]
\end{example}

\subsection{Commutators}\label{SH}
To discuss central extensions we require commutators: we compute using the Higgins commutator~\cite{MM-NC}. The characterisation of central extensions on which Definition~\ref{def:CentralExtension} is based works if binary Higgins commutators suffice to express higher centrality. This happens~\cite{RVdL3} whenever the \emph{Smith commutator} of equivalence relations~\cite{Smith} agrees with the \emph{Huq commutator} of normal subobjects~\cite{Huq}. Semiabelian categories satisfying this \defn{Smith is Huq} condition \defn{(SH)} include pointed strongly protomodular varieties~\cite{Borceux-Bourn} and \emph{categories of interest}~\cite{Orzech}. So categories of groups, Lie algebras and associative algebras are examples.

\subsection{Construction of the Higgins commutator}%
\label{Higgins}%
Given two monomorphisms \LM k\from{K\to X}\RM and \LM l\from{L\to X}\Rm, form the short exact sequence below from the comparison map \LM K+L\RM and \LM K\times L\Rm:
$$
\xymatrix@R=5ex@C=3em{
0 \ar[r] &
 K\diamond L \ar@{.{ >>}}[d] \ar@{{ |>}->}[r] &
 \CPrdct{K}{L} \ar@{-{ >>}}[r]^-{\left\links\begin{smallmatrix}\IdMap{K} &
 0 \\
0 &
 \IdMap{L}\end{smallmatrix}\right\rechts} \ar[d]^-{\left\links\begin{smallmatrix}k& l\end{smallmatrix}\right\rechts} &
 \Prdct{K}{L} \ar[r] & 0 \\
& [K,L] \ar@{{ >}.>}[r] & X
}
$$
The image $[K,L]\to X$ of \LM K\diamond L\RM in $X$ is called the \defn{Higgins commutator} of~$K$ and $L$. For groups or Lie algebras this yields the familiar commutator subobject.

An object $X$ semiabelian category $\X$ is abelian precisely when \LM [X,X]=0\Rm, so that \LM \ab_{\X}(X)\cong \frac{X}{[X,X]}\Rm.

\subsection{Birkhoff Subcategories}
\label{subsubsec:BirkhoffSubCats}%
A \defn{Birkhoff subcategory} of~$\X$ is a full and replete reflective subcategory which is closed under subobjects and regular quotients~\cite{Janelidze-Kelly}. A~Birkhoff subcategory of a variety of universal algebras is the same thing as a subvariety in the Birkhoff sense. For example, the abelian core \LM \AbCore{\X}\RM of a semiabelian category $\X$ is always a Birkhoff subcategory.

\section{Sequentially Right Exact Functors}
\label{sec:SequentiallyREFunctors}

As is well known, a functor between abelian categories preserves cokernels exactly when it preserves finite colimits, and so either property can be taken as a definition of `right exact functor'. For functors between semiabelian categories, the two properties differ. Moreover, subobjects need not be normal. So it is not immediately clear how to extend the notion of `right exact functor' to semiabelian categories.

Guided by the objective to facilitate functor derivation, we introduce here the concept of `sequentially right exact functor' between semiabelian categories; see~\ref{def:Sequentially(R)EFunctor} below. As deeper evidence that it serves its purpose well, we offer Theorem~\ref{thm:ACat-Reflective-In-SACat} which asserts that any such functor commutes with abelianisation up to unique natural equivalence. --- This property is not shared by the competitors cokernel preserving, respectively finite colimits preserving functors.

Let us begin this development by recalling the following concepts:
\begin{enumerate}[(1)]
\item A morphism \LM u\from U\to V\RM in a semiabelian category is called \defn{proper} if its image is a kernel.
\item The diagram below is \defn{exact at $V$}
$$
\xymatrix@R=5ex@C=3em{
U \ar[r]^-{u} &
 V \ar[r]^-{v} &
 W
}
$$
if \LM \Ker{v} = \Img{u}\Rm. This implies, in particular, that $u$ is a proper morphism.
\item An \defn{exact sequence} is a sequence of morphisms which is exact in each position. A \defn{short exact sequence} is an exact sequence
\begin{equation*}\label{Short-Exact-Sequence}
\xymatrix{0 \ar[r] & K \ar@{{ |>}->}[r]^-{k} & X \ar@{-{ >>}}[r]^-{f} & Y \ar[r] & 0.}
\end{equation*}
\end{enumerate}

\begin{definition}[Sequentially (right) exact functor]
\label{def:Sequentially(R)EFunctor}%
A functor \LM T\from \X\to \Y\RM between semiabelian categories is called \defn{sequentially exact} if it preserves all short exact sequences. It is called \defn{sequentially right exact} if it turns an exact sequence, as on the left below, into the exact sequence on the right.
\begin{equation*}\label{Right-Exact-Functor}
\xymatrix{
U \ar[r]^-{u} &
 X \ar@{-{ >>}}[r]^-{f} &
 Y \ar[r] &
 0
}
\qquad\qquad
\xymatrix{
TU \ar[r]^-{Tu} &
 TX \ar@{-{ >>}}[r]^-{Tf} &
 TY \ar[r] &
 0
}
\end{equation*}
\end{definition}

A sequentially right exact functor is automatically \defn{proper} in the sense that it sends proper morphisms in $\X$ to proper morphisms in $\Y$. One key feature of sequentially right exact functors is the following lemma; note that preservation of finite colimits need not suffice for it to hold.

\begin{lemma}
\label{lem:SRE->PreservesFiniteProducts}
A sequentially right exact functor between semiabelian categories preserves finite products.
\end{lemma}
\begin{proof}
In Definition~\ref{def:Sequentially(R)EFunctor} let $X$ be the product of $U$ and $Y$, $u=\langle 1_{U},0\rangle$ and $f=\pi_{Y}$. Then $Tu$ is a proper split monomorphism, which implies that $TX\cong TU\times TY$ by the Short Five Lemma.
\end{proof}

\subsection{The $2$-Category \LM \SACat\Rm}
\label{subsec:SACat}%
Here we introduce our work environment: let \LM \SACat\RM be the (large) $2$-category with objects semiabelian categories, morphisms sequentially right exact functors and $2$-cells natural transformations between them. It contains \LM \ACat\Rm, the (large) $2$-category of abelian categories and right exact functors as a full and replete subcategory. The following Theorem~\ref{thm:ACat-Reflective-In-SACat} binds \LM \SACat\RM to \LM \ACat\RM in a manner which is fundamental for our purposes.

\begin{theorem}
\label{thm:ACat-Reflective-In-SACat}
The $2$-category \LM \ACat\RM is pseudo-reflective~\cite{Fiore} in \LM \SACat\RM via the `abelian core' functor
$$
\AbCore{-}\from \SACat \longrightarrow \ACat .
$$
Its unit at a semiabelian category $\X$ is the reflector \LM\ab_{\X}\from {\X \to \AbCore{\X}}\RM as in~\ref{subsec:AbCore}.
\end{theorem}
\begin{proof}
Checking that for any abelian category $\A$, the functor
\[
\ab_{\X}^{*}=\NatCat(\ab_{\X},\A)\colon \NatCat(\Ab(\X),\A)\to \NatCat(\X,\A)
\]
defined by composition with $\ab_{\X}$ is an equivalence of categories amounts to proving that any sequentially right exact functor \Lm T\from {\X\to \Y}\RM in $\SACat$ commutes with abelianisation up to unique natural isomorphism: $\Ab(T)\comp \ab_{\X}\cong\ab_{\Y}\comp T$.

We know from~\ref{lem:SRE->PreservesFiniteProducts} that $T$ commutes with finite products. So it preserves the structure diagrams for an abelian group object, hence sends abelian group objects in $\X$ to abelian group objects $\Y$.

\emph{Step 1: suppose $T$ takes values in an abelian category.} Given an object $X$ in $\X$, consider the commutative diagram below.
$$
\xymatrix@R=5ex@C=3em{
& T(\CPrdct{X}{X}) \ar@{-{ >>}}[rr]^-{T\left\links\begin{smallmatrix}\IdMap{X} & 0\\
0&\IdMap{X}\end{smallmatrix}\right\rechts} \ar@{-}[d]^-{T\left\links\begin{smallmatrix}\IdMap{X}& \IdMap{X}\end{smallmatrix}\right\rechts} \ar@{-{ >>}}[ld] &&
 T(\Prdct{X}{X}) \ar@{-{ >>}}[dd] \\
\CPrdct{(TX)}{(TX)} \ar@{-{ >>}}[rr]_(.7){\cong} \ar@{-{ >>}}[dd] \ar@<5pt>[ru]^{\sigma} &
 {\ } \ar@{-{ >>}}[d] &
 \Prdct{(TX)}{(TX)} \ar@{-{ >>}}[dd] \ar[ru]_{\cong} \\
& TX \ar@{-}[r] & {\ } \ar@{-{ >>}}[r]^-{\AdjUnit[X]} &
 \pushout T(\ab_{\X}(X)) \\
TX \ar@{-{ >>}}[rr]_-{\AdjUnit[TX]}^-{\cong} \ar@{=}[ru] &&
\pushout \ab_{\Y}{(TX)} \ar[ru]_{\phi}
}
$$
The front face is the abelianisation pushout of \LM TX\RM in $\Y$. The back face results from applying $T$ to the abelianisation pushout of $X$ in $\X$. This is a pushout as well because right exactness of $T$ yields a regular epimorphism between the kernels of the horizontal morphisms. The morphism $\sigma$ is the canonical one. It has a left inverse because finite sums and products coincide in \LM \AbCore{\Y}\Rm. Consequently, $\phi$ is an isomorphism as required.

\emph{Step 2: \Lm T\from \X\to \Y\RM is arbitrary in \LM \SACat\Rm.} Then \LM \ab_{\Y}\comp T\from \X\to \AbCore{\Y}\RM is a functor in \LM \SACat\RM to which we may apply Step 1. Thus the claim follows.
\end{proof}

\begin{corollary}
A functor \LM T\from \X\to \A\RM from a semiabelian category to an abelian category is sequentially right exact if and only if it commutes with finite colimits.
\end{corollary}
\begin{proof}
Suppose $T$ is sequentially right exact. By~\ref{thm:ACat-Reflective-In-SACat}, $T$ is equivalent to the composite \LM \AbCore{T}\comp \ab_{\X}\Rm. The abelianisation functor on $\X$ has a right adjoint, hence commutes with arbitrary colimits. Moreover, the restriction \LM \AbCore{T}\RM of $T$ to the abelian core of $\X$ is sequentially right exact if and only if it commutes with finite colimits.

Conversely, if $T$ commutes with arbitrary finite colimits, then it commutes with cokernels, and this implies the claim.
\end{proof}

\subsection{Adjunctions in \LM \SACat\Rm}
\label{subsec:SACat-Adjunctions}%
Consider an adjoint functor pair in \LM \SACat\Rm, with left adjoint \LM T\from \X\to \Y\RM and right adjoint \LM G\from \Y\to \X\Rm. Then $G$ is automatically sequentially exact, and $T$ sends projectives in $\X$ to projectives in $\Y$. Furthermore, $G$ reflects regular epimorphisms if and only if the adjunction counit \LM \AdjCUnit[Y]\from TGY\to Y\RM is a regular epimorphism for every $Y$ in $\Y$. In this case, whenever $\X$ has enough projectives, so does $\Y$.

Examples of \LM \SACat\Rm-adjunctions in which $G$ reflects regular epimorphisms include all regular epi-reflections to full and replete semi-abelian subcategories of $\X$, as well as all `change of ring' adjunctions between module categories. Every reflector to an abelian subcategory of which the inclusion functor is exact is a member of an \LM \SACat\Rm-adjunction; likewise, such is the reflector to any subvariety of a semi-abelian variety.

Beware: even the category of abelian groups contains full and replete regular epi-reflective subcategories which are not part of an \LM \SACat\Rm-adjunction---for instance, the category of torsion free abelian groups is one such, as it is not abelian, and even fails to be semi-abelian. Examples of adjoint functor pairs between abelian categories which do not belong to \LM \SACat\RM can be found amongst the embeddings of categories of sheaves into categories of presheaves; see, for instance,~\cite[pp.~26ff.]{Weibel}.

\section{Yoneda's Isomorphism Commutes with Homology: \\ Semiabelian Case}
\label{sec:Yoneda-Commutes-Homology-SemiAb}

\subsection{The Homological Yoneda Lemma in Semiabelian Categories}
Here we show that the homological Yoneda lemma holds for (sequentially) right exact functors from a semiabelian category to the category of abelian groups; see~\ref{Lemma-Yoneda-Iso-Commutes-Homology-Semiab}. As an immediate consequence, we see that the abelian based formula of Yoneda~\cite{Yoneda-Exact-Sequences} and Hilton--Rees~\cite{Hilton-Rees} for derived functors is valid for semiabelian domain categories as well.

Let us begin by explaining how we derive functors here. Every object $X$ in a semiabelian category $\X$ with enough projectives has a (semi)simplicial projective resolution \LM P(X)\to X\Rm, and the left derived functors of a right exact \LM T\from \X\to \AbGrps\RM may be defined via the Moore homology
\[
\LDrvd{T}{n}{X} \DefEq \H_n T(P(X)).
\]
Similarly, for an object $A$ in the abelian core \LM \AbCore{\X}\RM of $\X$,
\[
\ExtFnctr{n}{X}{A} \DefEq \H^n \HomFnctr{\X}{P(X)}{A}
\]
is the $n$-th derived functor of the contravariant functor \LM \HomFnctr{\X}{\cdot}{A}\from{\X\to \AbGrps}\Rm. This slightly ad hoc approach to deriving functors is sufficient for our purposes. The reader is therefore free to adopt, for example,
\begin{enumerate}
\item {\em the framework of comonadically defined resolutions} that are available if $\X$ is a semiabelian variety: the comonad comes from the `(underlying set)--(free object)' adjunction;
\item {\em Quillen's framework of simplicial model categories} and choose a cofibrant replacement of a constant simplicial object, in the simplicial model category structure in Quillen's
~\cite[II, Theorem 4]{Quillen} via~\cite{VdLinden:Simp}.
\end{enumerate}
For example, if \LM \X=\Grps\RM and $T$ is abelianisation, then \LM \LDrvd{T}{n}{X} \cong \H_{n+1}(X)\Rm, where the right hand side is classical group homology.

\begin{lemma}[Homological Yoneda lemma: semiabelian case]%
\label{Lemma-Yoneda-Iso-Commutes-Homology-Semiab}%
Let \LM T\from \X\to \AbGrps\RM be a (sequentially) right exact functor from a semiabelian category $\X$ to the category of abelian groups, and let $S$ be a simplicial object in $\X$. Then, for \LM n\geq 0\Rm, there is an isomorphism, which is natural in $S$ and in $T$:
$$
\NatTrafos{\H^n\HomFnctr{\X}{S}{-}|_{\AbCore{\X} }}{T|_{\AbCore{\X}}} \overset{\cong}{\to} \H_nTS.
$$
\end{lemma}
\begin{proof}
Let $C$ denote the unnormalised chain complex associated to \LM \ab_{\X}(S)\Rm. Then
\begin{gather*}
\H^{n}\HomFnctr{\X}{S}{A} \cong \H^{n}\HomFnctr{\Ab(\X)}{\ab_{\X} (S)}{A} \cong \H^{n}\HomFnctr{\AbCore{\X}}{C}{A} \\
\H_{n}T(S)\cong\H_{n}T(\ab_{\X}(S)) \cong \H_{n}T|_{\AbCore{\X}}(C)
\end{gather*}
Critical here is that, up to a natural isomorphism, $T$ factors through the abelian core of $\X$; see~\ref{thm:ACat-Reflective-In-SACat}. Now the claim follows from Theorem~\ref{Lemma-Yoneda-Commutes-Homology}.
\end{proof}

\begin{theorem}\label{Theorem-Yoneda-Derived-Semiab}
Let \LM T\from \X\to \AbGrps\RM be a right exact functor on a semiabelian category with enough projectives. Then, for $X$ in $\X$ and \LM n\geq 0\Rm, there is an isomorphism
\[
\YY_{n}\from \NatTrafos{\ExtFnctr{n}{X}{-}}{T|_{\AbCore{\X}}} \overset{\cong}{\to} \LDrvd{T}{n}{X} ,
\]
which is natural in $X$ and in $T$.
\end{theorem}
\begin{proof}
For \LM n=0\Rm, this follows from Lemma~\ref{Lemma-YonedaRestriction}. For \LM n\geq 1\Rm, apply Theorem~\ref{Lemma-Yoneda-Iso-Commutes-Homology-Semiab} to a simplicial projective resolution of $X$.
\end{proof}

\begin{remark}
The isomorphism in Theorem~\ref{Theorem-Yoneda-Derived-Semiab} actually expresses a local self-adjointness property of the derivation operator on suitable functors,
\[
\NatTrafos{\ExtFnctr{n}{X}{-} }{T|_{\AbCore{\X}}}\cong \NatTrafos{\HomFnctr{\X}{ X }{ - } }{\L_nT}.
\]
\end{remark}

In Section~\ref{sec:UCEs} we will require the following consequences of~\ref{Lemma-Yoneda-Iso-Commutes-Homology-Semiab} and~\ref{Theorem-Yoneda-Derived-Semiab}.

\begin{corollary}\label{Cory-Yoneda-Derived-SemiAb-RMod}
Let \LM T\from \X\to \LModules{R}\RM be a right exact functor. If $\X$ has enough projectives then, for \LM n\geq 0\RM and $X$ in $\X$, there is an isomorphism of left $R$-modules
\[
\YY_{n}\from\NatTrafos{\ExtFnctr{n}{X}{-}}{T|_{\AbCore{\X}}} \overset{\cong}{\to} \LDrvd{T}{n}{X},
\]
which is natural in $X$ and in $T$.
\end{corollary}
\begin{proof}
An isomorphism of right $\ZZ$-modules comes from~\ref{Theorem-Yoneda-Derived-Semiab}. For the additional enrichment in left $R$-modules, see Remark~\ref{rem:Yoneda-Iso-Commutes-Homology-RMod}.
\end{proof}

We omit the (rather lengthy) proof of the next result which is not used in what follows.

\begin{corollary}
\label{Cory-Yoneda-Derived-SemiAb-Adjoint}%
Let \LM T\from \X\to \R\RM be a (right exact) reflector from a semiabelian variety $\X$ to a full and replete subcategory $\R$ of its abelian core. Then for \LM n\geq 0\RM and $X$ in $\X$, there is an isomorphism of $\R$-objects
\[
\YY_{n}\from \NatTrafos{\ExtFnctr{n}{X}{-}}{T|_{\AbCore{\X}}}|_{\R} \overset{\cong}{\to} \LDrvd{T}{n}{X},
\]
which is natural in $X$ and in $T$.
\end{corollary}

\subsection{A Long Exact Homology Sequence, Semiabelian Case}\label{LES-SemiAbelian}
Given a semiabelian category $\X$ with enough projectives, let \LM T\from {\X\to \AbGrps}\RM be a right exact functor, and consider a short exact sequence \LM 0 \to K\to X\to Y\to 0\RM in $\X$. If $\X$ is abelian, there is the familiar associated long exact sequence of derived functors. However, the example of \LM \X=\Grps\RM shows that, for general $\X$, the relationship between the derived functors of $K$, $X$, and $Y$ can be much more complicated.

On the other hand, we can say more in the special case where $T$ is the reflector from $\X$ onto a reflective subcategory $\R$ of the abelian core of $\X$. Given an \defn{extension (over $Y$)}, that is, a regular epimorphism $f\colon{X\to Y}$, choose a kernel~$K$ and divide out the Higgins commutator (\ref{Higgins}) to obtain the central extension
$$
\xymatrix@R=5ex@C=3em{
A\DefEq K/[K,X] \ar@{{ |>}->}[r] &
 X/[K,X] \ar@{-{ >>}}[r] &
 Y.
}
$$
Then construct \LM T_1(f)\RM via the pushout of this central extension along the reflection unit \LM A\to T(A)\EqDef\Ell_0(f)\Rm. This yields a reflector \LM T_1\from {\CatExt(\X)\to \CExt_{\R}(\X)}\Rm, where \LM \CatExt(\X)\RM is the category of regular epimorphisms in $\X$, and \LM \CExt_{\R}(\X)\RM is the full subcategory of regular epimorphisms with central kernel, and whose kernel object is in~$\R$. In this setting the derived functors of \LM T_1\RM are defined (up to a unique natural isomorphism). Applied to $f$, the functor \Lm \L_{n}T_1\RM gives a central extension of the form
$$
\xymatrix@R=5ex@C=3em{
\Ell_n(f) \ar@{{ |>}->}[r]^{\cong} &
 \Ell_n(f) \ar@{-{ >>}}[r]^-{\LDrvd{T_1}{n}{f}} &
 0.
}
$$
Via a result in~\cite{Everaert-Gran-TT}, the proof given in Everaert~\cite{Tomasthesis}, see also~\cite[2.6]{GVdL2}, establishes a long exact sequence in~$\R$:
\[
\xymatrix@R=.3ex{
{\cdots} \ar[r] &
 \LDrvd{T}{n+1}{Y} \ar[r]^-{\delta^{n+1}_{f}} &
 \Ell_n(f) \ar[r]^-{\gamma^{n}_{f}} &
 \LDrvd{T}{n}{X} \ar[r]^-{f_*} &
 \LDrvd{T}{n}{Y} \ar[r] & \cdots \\
{\cdots} \ar[r] &
 \LDrvd{T}{1}{Y} \ar[r]_-{\delta^{1}_{f}} &
 \Ell_0(f) \ar[r]_-{\gamma^{0}_{f}} &
 \LDrvd{T}{0}{X} \ar[r]_-{f_*} &
 \LDrvd{T}{0}{Y} \ar[r] &
 0. }
\]
A key issue here (which will appear again in Remark~\ref{Remark-Terminology}) is the coincidence of central extensions in the above sense with the central extensions relative to the reflector $T$ defined via categorical Galois theory in~\cite{Everaert-Gran-TT}.

\subsection{A Long Exact Cohomology Sequence}
A similar result holds for cohomology. Let \LM P(f)\to f\RM be a simplicial projective resolution of $f$, and write
\[
\mathcal{E}\!\mathit{xt}_{1}^{n}(f,A)\DefEq \H^{n}\HomFnctr{\X}{ \Ell_0(P(f)) }{ A }
\]
for the cohomology of $f$ with coefficients in an abelian object $A$ of $\X$. Using results in~\cite{EGVdL}, we find:

\begin{proposition}\label{LES-Cohomology}
Let \LM f\from {X\to Y}\RM be a regular epimorphism in a semiabelian category $\X$ with enough projectives, and let $A$ be an abelian object. Then the sequence of abelian groups below is exact:
\[
\resizebox{\textwidth}{!}
{\xymatrix@C=3em@R=1ex{
{\cdots} &
 \Ext^{n+1}(Y,A) \ar[l] &
 \mathcal{E}\!\mathit{xt}_{1}^{n}(f,A) \ar[l]_-{\varepsilon^{n+1}_{f}} &
 \Ext^{n}(X,A) \ar[l]_-{\varphi^{n}_{f}} &
 \Ext^{n}(Y,A) \ar[l]_-{\Ext^{n}(f,A)} &
 \cdots \ar[l] \\
{\cdots} &
 \Ext^{1}(Y,A) \ar[l] &
 \mathcal{E}\!\mathit{xt}_{1}^{0}(f,A) \ar[l]^-{\varepsilon^{1}_{f}} &
 \Ext^{0}(X,A) \ar[l]^-{\varphi^{0}_{f}} &
 \Ext^{0}(Y,A) \ar[l]^-{\Ext^{0}(f,A)} &
 0 \ar[l]}}
\]
\end{proposition}

\begin{theorem}\label{Theorem-Duality-NonAbelian}
Given a semiabelian variety $\X$, consider a reflective subcategory~$\R$ of its abelian core with reflector \LM T\from \X\to \R\Rm. Then a regular epimorphism~\LM f\from{X\to Y}\RM in~$\X$ has the associated commutative ladder of $\R$-objects below.
\[
\resizebox{\textwidth}{!}
{\xymatrix@R=8ex@C=1.5em{
\NatTrafos{ \ExtFnctr{n+1}{Y}{-} }{ T|_{\Ab(\X)} }|_{\R} \ar@<+10pt>[d]_{\cong}^{\YY_{n+1}(Y)} \ar[r]^-{(\varepsilon^{n}_{f})^\asterisk} &
 \NatTrafos{ \ExtFnctr[1]{n}{f}{-} }{ T|_{\Ab(\X)}}|_{\R} \ar[r]^{(\varphi^{n-1}_{f})^{\asterisk}} \ar[d]^{\cong} &
 \NatTrafos{ \ExtFnctr{n}{X}{-} }{ T|_{\Ab(\X)} }|_{\R} \ar[r]^-{(f^\asterisk)^\asterisk} \ar[d]_{\YY_{n}(X)}^{\cong} &
 \NatTrafos{ \ExtFnctr{n}{Y}{-} }{ T|_{\Ab(\X)} }|_{\R} \ar@<-10pt>[d]_{\YY_{n}(Y)}^{\cong} \\
\cdots \to \LDrvd{T}{n+1}{Y} \ar[r]_-{\delta^{n+1}_{f}} &
 \Ell_n(f) \ar[r]_-{\gamma^{n}_f} &
 \LDrvd{T}{n}{X} \ar[r]_-{\LDrvd{T}{n}{f}} &
 \LDrvd{T}{n}{Y}\to \cdots}}
\]
\end{theorem}
\begin{proof}
Apply $T$ to a simplicial projective resolution \LM p\from P\to Q\RM of $f$ to obtain a regular epimorphism of abelian simplicial objects. There is an associated short exact sequence of chain complexes to which we apply Theorem~\ref{Theorem-Yoneda-Connecting-Map}. With the identification of the homology of the kernel chain complex as the derived functors of \LM T_1\RM in~\cite{Tomasthesis} the proof is complete.
\end{proof}

\section{Cohomology, Higher Dimensional Central Extensions,\\ and Satellites} %
\label{sec:HigherExtensions}%

In an abelian environment, Yoneda~\cite{Yoneda-Exact-Sequences} achieved an interpretation of the groups \LM \Ext^n(X,A)\RM in terms of congruence classes of $n$-step extensions:
$$
A \longrightarrow X_n \to \cdots \to X_1 \longrightarrow X.
$$
Here we explain a semiabelian analogue of Yoneda's interpretation of Ext; see~\cite{RVdL2} for details. It is based on categorical Galois theory and the concept of higher central extension. In the case of groups, an alternative interpretation of the Ext groups in terms of \emph{crossed extensions} is available via work of ~\cite{Holt, Huebschmann}. However, it is not clear at this time whether a crossed interpretation of extensions extends to semiabelian categories.

Then we use Yoneda's view of Ext to relate the homological Yoneda lemma to the satellites view of functor derivation.

This section also provides background needed in the following section on universal central extensions.

\subsection{Extensions}
\label{subsec:Extensions}
Let $\ThreeCat$ denote the category: \LM 0 \rightarrow 1 \rightarrow 2\Rm. For \LM n\geq 1\Rm, the category \LM \ThreeCat^n\RM has the initial object \LM i_n\DefEq (0,\dots,0)\RM and the terminal object \LM t_n\DefEq (2,\dots ,2)\Rm. Moreover, it has an embedding \LM \alpha_{e,i}\from \ThreeCat\to \ThreeCat^n\RM parallel to the $i$-th coordinate axis, for each object $e$ whose $i$-th coordinate is $0$.

Now, given objects $X$ and $A$ in $\X$, an \defn{$n$-extension}~\cite{EGVdL, EGoeVdL} \defn{under $A$ and over~$X$} in $\X$ is a functor \LM E\from \ThreeCat^n\to \X\RM which sends \LM i_n\RM to $A$, \Lm t_n\RM to $X$, so that each composite below is a short exact sequence:
$$
\xymatrix@R=5ex@C=3em{
\ThreeCat \ar[r]^-{\alpha_{e,i}} &
 \ThreeCat^n \ar[r]^-{E} &
 \X
}
$$
For example, a $1$-extension under $A$ and over $X$ is just a short exact sequence \LM A=E_0 \to E_1 \to E_2=X\Rm. A $2$-extension under $A$ and over $X$ is a \emph{$3\times 3$ diagram}, in which each row and column is short exact:
$$
\xymatrix@R=5ex@C=3em{
A=E_{0,0} \ar@{{ |>}->}[r] \ar@{{ |>}->}[d] &
 E_{1,0} \ar@{-{ >>}}[r] \ar@{{ |>}->}[d] &
 E_{2,0} \ar@{{ |>}->}[d] \\
E_{0,1} \ar@{{ |>}->}[r] \ar@{-{ >>}}[d] &
 E_{1,1} \ar@{-{ >>}}[r] \ar@{-{ >>}}[d] &
 E_{2,1} \ar@{-{ >>}}[d] \\
E_{0,2} \ar@{{ |>}->}[r] &
 E_{1,2} \ar@{-{ >>}}[r] &
 X=E_{2,2}
}
$$

\subsection{Centrality for Higher Extensions}
\label{subsec:Centrality}
Following the ideas in~\cite{Janelidze:Hopf-talk, Janelidze:Double}, higher central extensions were originally defined through categorical Galois theory in~\cite{EGVdL} with the aim of extending the Brown--Ellis--Hopf formulae~\cite{Brown-Ellis} to a categorical context. Here we use an alternate approach based on the work of Rodelo and Van~der Linden~\cite{RVdL3}, valid in semiabelian categories with enough projectives which satisfy the \emph{Smith is Huq} condition~\ref{SH}. It expresses the centrality condition from categorical Galois theory in terms of Higgins commutator (\ref{Higgins}) properties of a given extension diagram: For a subset \LM R\subseteq n\DefEq \{ 0, \dots , n-1 \} \Rm, let
$$
E(R) \DefEq E_{\varepsilon_0,\dots ,\varepsilon_{n-1}}
\qquad\text{where}\qquad
\varepsilon_i \DefEq
\begin{cases}
1 & \text{if \ $i\in R$} \\
0 & \text{if \ $i\not\in R$.}
\end{cases}
$$

\begin{definition}
\label{def:CentralExtension}%
An $n$-extension $E$ is \defn{central (with respect to abelianisation)} if, for each \LM R\subseteq n\Rm, the Higgins commutator (\ref{Higgins}) of the monomorphisms \LM {E(R)\to E(n) \leftarrow E(S)}\RM vanishes; here $S$ is the complement of $R$ in $n$.
\end{definition}

So, for example, choosing \LM R\DefEq \emptyset\Rm, we see that \LM A=E(\emptyset)\RM is central in \LM E(n)\RM and, hence, in \LM E(R)\Rm, for every \LM R\subseteq n\Rm. In particular, it is an abelian object.

\subsection{Groups of Central $n$-Extensions}
\label{subsec:Groups-n-Extensions}%
As in Yoneda's theory of equivalence classes of $n$-step extensions, we now define a central $n$-extension $E$, under $A$ and over~$X$, to be \defn{congruent} to another such extension $E'$ if there is a natural transformation \LM \tau\from E\to E'\RM which is the identity on $A$ and on $X$.

`Congruence' generates an equivalence relation on central $n$-extensions under~$A$ and over $X$; let \LM \Centr^{n}(X,A)\RM denote the resulting set of equivalence classes. The main results of~\cite{RVdL2, RVdL3}, based on torsor theory, combined provide a natural isomorphism %
\label{iso:CExt=Ext}%
$$
\phi\from \Ext^{n}(X,A) \overset{\cong}{\longrightarrow} \Centr^{n}(X,A)
$$
whenever $\X$ is a semiabelian category with enough projectives satisfying the \emph{Smith is Huq} condition (see~\ref{SH}).

\begin{remark}\label{Rem:CubicalViewpoint}
The term `extension' has been used with different meanings in previous publications related to this article. Beginning with \cite{Janelidze:Double} and later \cite{Donadze-Inassaridze-Porter,EGVdL}, the goal was to find a higher analogue for a regular epimorphism with central kernel. The outcome was termed `$n$-fold central extension'. Subsequently, work toward defining congruence classes of higher central extensions which correspond to cohomology required a refinement of the concept. Such a refinement is proposed in~\cite{RVdL2}. Accordingly, we now use the term `$n$-fold central extension' exclusively for a diagram as defined in \ref{def:CentralExtension}. Such a diagram contains a cube of regular epimorphisms, namely the restriction of \LM E\from \ThreeCat^n\to \X\RM to $\{1,2\}^{n}$. For clarity, following \cite{RVdL2}, we now refer to this cube of regular epimorphisms as an \defn{$n$-cubical extension}.

Thus, with terminology as explained, an $n$-fold extension as defined in \ref{def:CentralExtension} is central  if and only if its cube of regular epimorphisms is central in the sense of categorical Galois theory (with respect to abelianisation).
\end{remark}

\subsection{Satellites}
\label{subsec:Satellites}%
Now let $\X$ be a semiabelian category with enough projectives that satisfies \emph{Smith is Huq}~\eqref{SH}.

Here is a satellites view of the relationship between homology and cohomology established by the Yoneda isomorphism; compare ~\cite{Yoneda-Exact-Sequences, Mitchell:Categories, Janelidze-Satellites, Guitart-Bril} in the abelian case. Sending a central $n$-extension under $A$ and over $X$ to the pair \LM (A,X)\RM yields the span
\begin{equation}\label{D}\tag{$\dagger$}
\xymatrix@C=4em{
\X &
 \CExt^{n}(\X) \ar[l]_-{\text{eval}_{t_n}}^-{\CC} \ar[r]^-{\text{eval}_{i_n}}_-{\D} &
 \AbCore{\X},
}
\end{equation}
where \LM \CExt^n(\X)\RM denotes the category of $n$-fold central extensions in $\X$. The connected components of the fibre of this span form the functor
\[
\Centr^{n}(\cdot,-)\colon \X^{\op}\times \Ab(\X)\to \Ab\colon (X,A)\mapsto \Centr^{n}(X,A),
\]
and these are naturally equivalent to cohomology in the guise of \LM \Ext^n(X,A)\RM by~\cite{RVdL2}. On the other hand, as shown in~\cite{GVdL2}, the $n$-th derived functor $\L_{n}\ab_{\X}\colon\X\to \Ab(\X)$ of the abelianisation functor is the right Kan extension of $\D$ along $\CC$.

We shall now connect this viewpoint with the homological Yoneda isomorphism. Let the example of the category of groups suffice to explain the essence.

\begin{example}[Groups]
Let \LM \ab\from {\Gp\to \Ab}\RM be the abelianisation functor. For a group $X$ and \LM n\geq 1\Rm, the homological Yoneda isomorphism~\ref{Theorem-Yoneda-Derived-Semiab} specialises to
\[
\Nat(\Centr^{n}(X,-),1_\Ab)\cong \H_{n+1}X.
\]
Explicitly, any element $x$ of $\H_{n+1}X$ corresponds to a collection of group homomorphisms, natural in $A$:
\[
\varphi=(\varphi_{A}\colon \Centr^{n}(X,A)\to A)_{A\in\Ob{\Ab}}.
\]
How does this result relate to the satellites interpretation of homology? Corollary~4.10 in~\cite{GVdL2} tells us that the Kan extension $\H_{n+1}X$ can be calculated as the limit of the (large) diagram \LM \D\colon \CExt^n_X(\Gp)\to \Ab\Rm, see~\eqref{D} above. Following Theorem~V.1 in~\cite{MacLane}, this limit may in turn by computed in such a way that an element of its underlying set is given by a cone from the one-point set $*$ to $\D$. That is to say, an element of $\H_{n+1}X$ is determined by a (compatible) choice $\lambda=(\lambda_F)_F$ of elements~$\lambda_F$ of~$\D(F)$, one for each $n$-extension $F$ over $X$, such that $\D(f)(\lambda_F)=\lambda_G$ whenever $f\colon {F\to G}$ is a morphism of $n$-extensions over~$X$. On the other hand, a natural transformation such as~$\varphi$ above is given by the choice of an element $\varphi_A([F])\in A$ for each equivalence class $[F]$ of a central extension~$F$ under~$A$ and over $X$, again in the suitably compatible way. Now of course $A=\D(F)$, so to see that the two types of choice correspond to one another, it suffices that the equivalence classes in $\Centr^{n}(X,A)$ are compatible with the requirements on~$\lambda$. This is clear, because two central extensions $F$ and $G$ under $A$ and over $X$ are congruent when there exists a zigzag between them of which all arrows restrict to $1_{A}$ on the initial object and $1_{X}$ on the terminal object.
\end{example}

\begin{example}[Lie algebras]\label{Example-Lie}
A similar result holds for Lie algebras over a ring~$R$. Such a Lie algebra is abelian if and only if its Lie bracket (= Higgins commutator) is trivial, so ${}_{R}\Lie$ contains ${}_R\Mod=\Ab({}_{R}\Lie)$ as a Birkhoff subcategory. We again take $T$ equal to the abelianisation functor and obtain an isomorphism
\[
\Nat(\Centr^{n}(X,-),\ab_{{}_{R}\Lie})\cong \H_{n+1}X.
\]
\end{example}

\section{Universal Central Extensions}
\label{sec:UCEs}%

Let $\X$ be a semiabelian category with enough projectives, and let \LM T\from \X\to \R\RM be a reflector onto a subcategory which is contained in the abelian core of $\X$. A universal (central) extension over an object $X$ is a universal element (in the sense of~\cite[Section~III]{MacLane}) of an Ext functor
\[
\ExtFnctr{n}{X}{-}|_\R\from \R \to \AbGrps.
\]
In the present section we investigate under which circumstances such a universal extension over an object exists and, when it does, how to construct it.

We need to address the fact that there is a fundamental difference between the definition of universality for $1$-extensions and that for higher extensions: a central $1$-extension over $X$ is universal if it is an initial object in the category of all central $1$-extensions over $X$. On the other hand, in~\ref{Classical} we show that a double central extension over an object~$X$ which is initial amongst such is necessarily the zero double extension over \LM {X=0}\Rm. Our approach avoids this problem by defining universality of a central $n$-extension $U$ in terms of a property of the cohomology class \LM [U]\RM which it represents: \Lm [U]\RM is a representing object for \LM \ExtFnctr{n}{X}{-}|_\R\Rm. For \LM n\geq 1\Rm, we show that a universal $n$-fold central extension $U$ over $X$ and under \LM \LDrvd{T}{n}{X}\RM exists as soon as \LM \LDrvd{T}{n-1}{X}=\cdots=\LDrvd{T}{1}{X}=\LDrvd{T}{0}{X}=0\Rm.

\subsection{Orthogonal Pairs of Subcategories} In a pointed category $\C$ we take \LM \HomFnctr{\C}{X}{F}=0\RM to mean that $X$ is \defn{left orthogonal} to $F$, and that $F$ is \defn{right orthogonal} to $X$. We write \LM X\bot F\RM and \LM X\in {}^{\bot}F\RM or, equivalently, \Lm X^{\bot}\ni F\Rm. We always have $\bot\bot\bot=\bot$.

A pair of full subcategories \LM (\T,\F)\RM of $\C$ is called an \defn{orthogonal pair} if the following orthogonal complement relations hold:
\[
\T = {}^{\bot}\F \quad\text{and}\quad \T^{\bot}=\F.
\]
In a category of modules over a ring, an orthogonal pair of subcategories always forms a torsion theory. In a semiabelian category, \LM (\T,\F)\RM is a torsion theory precisely when $\F$ is reflective in $\C$~\cite{Bourn-Gran-Torsion}.

\subsection{Orthogonal Pairs in a Semiabelian Category}\label{Setting-Orthogonal-Pairs}
Now let $\X$ be a semiabelian category with enough projectives, and let \LM T\from \X\to \R\RM be a reflector onto a subcategory which is contained in the abelian core of $\X$. Then $T$ determines an orthogonal pair \LM (\T,\F)\RM in $\X$ by setting
\[
\T\DefEq{}^{\bot}\R \quad\text{and}\quad \F\coloneq \T^{\bot}.
\]
This is so because \LM \HomFnctr{\X}{X}{M} \cong \HomFnctr{\R}{TX}{M}\RM for all $M$ in $\R$.

\subsection{Acyclicity Classes of Order $n$}
For \LM n\geq 0\Rm, an object $X$ is called \defn{\LM n\Rm-acyclic (with respect to $T$)} if and only if
\[
\LDrvd{T}{0}{X} = \LDrvd{T}{1}{X} = \cdots = \LDrvd{T}{n}{X} = 0.
\]
Let \LM \TrsnCat[n]\RM be the full subcategory of $\X$ containing the $n$-acyclic objects.

\begin{theorem}\label{theorem:XinT_i<->Ext^i(X,LY)=0}
Let \LM T\from \X\to \R\RM be a reflector from a semiabelian variety~$\X$ to a full and replete subcategory $\R$ of its abelian core $\Ab(\X)$. Then an object $X$ of~$\X$ belongs to \LM \TrsnCat[n]\RM if and only if
\[
\ExtFnctr{0}{X}{M} = \cdots = \ExtFnctr{n}{X}{M} = 0\qquad \text{for each $M$ in $\R$.}
\]
\end{theorem}
\begin{proof}
Suppose $X$ belongs to \LM \TrsnCat[n]\Rm, and let \LM P\to X\RM be a simplicial projective resolution of $X$. Then the normalised chain complex $C$ associated to \LM TP\RM is exact in positions \LM 0\leq i\leq n\Rm, augmented by $0$, and positionwise projective in $\R$; see~\ref{subsec:SACat-Adjunctions}. Thus the claim follows from Lemma~\ref{lem:AbelianUCT}.

Now suppose \LM \ExtFnctr{i}{X}{M}=0\RM for all $M$ in $\R$, and \LM 0\leq i\leq n\Rm. To see that \LM \LDrvd{T}{i}{X}\RM vanishes for all \LM 0\leq i\leq n\Rm, we use the higher Yoneda isomorphism  (Theorem~\ref{Cory-Yoneda-Derived-SemiAb-RMod}). Indeed, the abelian group objects of the semiabelian variety $\XX$ have underlying abelian groups, and so:
\[
\NatTrafos{\ExtFnctr{i}{X}{-}}{T|_{\Ab(\X)}} \overset{\cong}{\longrightarrow} \LDrvd{T}{i}{X}.
\]
We claim that the term on the left contains only the zero-transformation. Indeed, if \LM {\tau\from
\ExtFnctr{i}{X}{-}\To T|_{\AbCore{\X}} }\RM is a natural transformation, consider its effect on an arbitrary~$A$ in \LM\Ab(\X)\Rm. The reflection unit \LM \vartheta_{A}\from A\to TA\RM gives the commutative diagram
\[
\xymatrix@R=5ex@C=3em{
\ExtFnctr{i}{X}{A} \ar[r]^-{\tau_A} \ar[d]_{(\vartheta_{A})_\asterisk} &
 TA \ar[d]^{\vartheta_{TA} = T\vartheta_A}_{\cong} \\
\ExtFnctr{i}{X}{TA} \ar[r]_-{\tau_{TA}} &
 T^2A
}
\]
By assumption \LM \ExtFnctr{i}{X}{TA}=0\Rm. So \LM \tau_A=0\Rm, and \LM \LDrvd{T}{i}{X}=0\RM follows.
\end{proof}

\begin{theorem}[Universal coefficient theorem]\label{Universal coefficient theorem}
Let \LM T\from \X\to \R\RM be a reflector from a semiabelian variety $\X$ to a full and replete subcategory of its abelian core. If $X$ in $\X$ is \LM (n-1)\Rm-acyclic, then the functors
\[
\HomFnctr{\R}{\LDrvd{T}{n}{X}}{-}\ ,\ \ExtFnctr{n}{X}{-}|_\R\from \R \to \AbGrps
\]
are naturally equivalent.
\end{theorem}
\begin{proof}
Let \LM P\to X\RM be a simplicial projective resolution of $X$ in $\X$. Then the chain complex $C$ associated to \LM TP\RM is augmented by $0$, positionwise projective in $\R$, and exact in positions \LM 0\leq i\leq n-1\Rm. It satisfies \LM \LDrvd{T}{n}{X}=\H_nC\RM and \LM \ExtFnctr{n}{X}{M}=\H^n(C,M)\Rm. Now the claim follows from Lemma~\ref{lem:AbelianUCT}.
\end{proof}

\begin{definition}
Let $\R$ be a reflective subcategory of a semiabelian category~$\X$ which is contained in \LM \AbCore{\X}\Rm. For $X$ in $\X$, and $M$ in $\R$, an element~$U$ of \LM \ExtFnctr{n}{X}{M}\RM is called a \defn{$\R$-universal} if the natural transformation
\[
\upsilon^{U}\from \HomFnctr{\R}{ M }{ - }\To \Ext^{n}(X,-)|_{\R}\from \R \longrightarrow \AbGrps
\]
determined by \LM \upsilon^{U}_{N}(f\colon M\to N) \DefEq\ExtFnctr{n}{X}{f}(U)\RM is a natural equivalence.
\end{definition}

See Section~III of~\cite{MacLane} for a detailed account of the relationship between universal elements and representable functors.

\begin{theorem}[Existence of universal central extensions]%
\label{thm:UCEsExst}%
Let $\X$ be a semiabelian variety satisfying the \emph{Smith is Huq} condition~\ref{SH}, and let $\R$ be a full and replete reflective subcategory of its abelian core. Then every object $X$ in~\LM \TrsnCat[n-1]\RM has a universal $\R$-central $n$-extension under \LM \LDrvd{T}{n}{X}\RM and over $X$.
\end{theorem}
\begin{proof}
Let \LM T\from \X\to \R\RM be a reflector. The element
$$
U \in \Centr^{n}(X,\LDrvd{T}{n}{X}) \overset{\ref{iso:CExt=Ext}}{\cong} \Ext^n(X,\LDrvd{T}{n}{X}) \overset{\ref{Universal coefficient theorem}}{\cong} \HomFnctr{\R}{\LDrvd{T}{n}{X}}{\LDrvd{T}{n}{X}}
$$
corresponding to \LM \IdMap{\LDrvd{T}{n}{X}}\RM is a universal $\R$-central $n$-extension.
\end{proof}

\begin{remark}
We say that an $\R$-central $n$-extension is \defn{universal} if it represents a universal $\R$-central $n$-extension class.
\end{remark}

\begin{remark}\label{Remark-Terminology}
Like in the one-dimensional case (see the last sentence of~\ref{LES-SemiAbelian}), the terminology ``$\R$-central $n$-extension'' used here agrees with the concept of ``$n$-extension in $\X$, central with respect to the reflective subcategory $\R$'' from categorical Galois theory. Indeed, in~\cite{Everaert-Gran-TT} the latter are characterised as central $n$-extensions $E$ in the sense of~\ref{subsec:Groups-n-Extensions} of which the initial object $\D(E)$ lies in~$\R$.
\end{remark}

\begin{example}
Let us interpret Theorem~\ref{thm:UCEsExst} in the case where \LM \X=\Grps\RM is the category of groups, and the reflector \LM T\from \Grps\to \AbGrps\EqDef\R \RM is abelianisation: If \LM n\geq 1\Rm, then every group $X$ with \LM \H_{n}(X,\ZNr)=\dots=\H_{1}(X,\ZNr)=0\RM has an \LM\AbGrps\Rm-universal central $n$-extension $U$. (Note the dimension shift here: \LM \LDrvd{\ab_{\Gp}}{n}{X} \cong \H_{n+1}(X, \ZNr)\Rm.) --- In other words, $U$ is a congruence class of $n$-central extensions under \LM \H_{n+1}(X,\ZNr)\RM and over $X$, and for every $n$-central extension $E$ under an abelian group $A$ and over~$X$ there exists a unique morphism \LM f\from \H_{n+1}(X,\ZNr)\to A\RM such that \LM f_*(U)\RM represents the congruence class of~$E$.

In \LM \Grps\RM the representability of \LM \H^{n+1}(X,\ZNr)\RM is already assured if \LM \H_n(X,\ZNr)\RM is torsion free; this follows from the universal coefficient theorem for chain complexes of free $\ZNr$-modules. It means that the general theory of universal $\R$-central extensions developed here provides sufficient existence conditions which may be stronger than necessary in particular semiabelian varieties.

The interpretation of Theorem~\ref{thm:UCEsExst} in the case where $\X$ is the category of Lie algebras over some commutative unital ring $R$ is similar.
\end{example}

\subsection{How to Construct Universal Central Extensions}\label{Construction}
In what follows we assume that $\R$ is a Birkhoff subcategory of $\Ab(\X)$---in other words, $\R$ is an abelian subvariety of $\X$. For any object $X$ of~$\X$ we write the kernel of the (regular epimorphic) unit ${X\to TX}$ of the reflection from~$\X$ to~$\R$ as a commutator $[X,X]_{\R}$.

Furthermore, we adopt the cubical view on $n$-fold extensions of~\cite{EGVdL} recalled in Remark~\ref{Rem:CubicalViewpoint}. An \defn{$n$-fold presentation} of an object $X$ is an $n$-fold cubic extension over $X$, see \ref{Rem:CubicalViewpoint}, in which all regular epimorphisms have a projective domain. It may be obtained from the truncation of a simplicial projective resolution in position \LM (n-1)\Rm.

We are going to prove that, given an $(n-1)$-acyclic object $X$, its universal $n$-fold central extension may be constructed as follows:
\begin{enumerate}
\item consider an $n$-fold presentation $F$ of $X$;
\item centralise $F$ to obtain the $n$-cubic central extension  \LM T_{n}F\Rm;
\item take commutators pointwise to obtain the $n$-fold arrow  \LM [T_{n}F,T_{n}F]_{\R}\Rm;
\item take kernels to obtain a  \LM \langle 3\rangle ^n\Rm-diagram $U$.
\end{enumerate}
It turns out that $U$ is weakly initial among all $n$-fold central extension over~$X$. Moreover, its direction is \LM \LDrvd{T}{n}{X}\Rm. It follows that the cohomology class represented by $U$ is indeed a universal central extension of $X$. To see this, the main difficulty lies in proving the following key lemma.

\begin{lemma}\label{Key Lemma}
\LM [T_{n}F,T_{n}F]_{\R}\RM is an $n$-cubic extension.
\end{lemma}

It is an immediate consequence of the following general result which relates higher acyclicity with properties of higher extensions.

\begin{proposition}
Let $C$ be an $n$-presentation of an $(n-1)$-acyclic object $X$ in~$\X$. Then the diagram $C\to TC$ in $\X$ is an  \LM (n+1)\Rm-cubic extension over $0$.
\end{proposition}
\begin{proof}
The object $X$ being $0$-acyclic means precisely that $C\to TC$ is an $(n+1)$-fold arrow over $0=TX$.

We prove that $C\to TC$ is an $(n+1)$-cubic extension by induction on $n$. For $n=1$ the diagram $C\to TC$ is a double cubic extension in $\X$ by sequential right exactness of~$T$ and the fact that in any semi-abelian category, all pushout squares of regular epimorphisms are regular pushouts.

Now suppose that the property holds for all $1\leq k < n$ and let $C$ be an $n$-cubic extension. The strong $\mathcal{E}$-Birkhoff property~\cite[Definition~2.5, Lemma~4.3]{EGVdL} tells us that $C\to T_{n}C$ is an ${(n+1)}$-cubic extension. On the other hand, by the Hopf formula, $n$-acyclicity of $C$ means precisely that $T_{n}C$ is the trivialisation of~$C$: see Remark~5.12 in~\cite{EverHopf}. The induction hypothesis, however, tells us that then $T_{n}C\to TC$ is an $(n+1)$-cubic extension. The result now follows because ${(n+1)}$-cubic extensions are closed under composition.
\end{proof}

\begin{proof}[Proof of Lemma~\ref{Key Lemma}]
By Proposition~3.9 in~\cite{EGVdL}, it suffices to note that $[C,C]_{\R}$ is the kernel of $C\to TC$.
\end{proof}

This easily gives us:

\begin{proposition}
$U$ is weakly initial among the $n$-fold central extensions over~$X$.
\end{proposition}
\begin{proof}
Since $F$ is a weakly initial cubic extension, its centralisation $T_{n}F$ is weakly initial central. The $n$-fold arrow \LM [T_{n}F,T_{n}F]_{\R}\Rm, being a sub--cubic extension of $T_{n}F$ by Lemma~\ref{Key Lemma}, is also weakly initial central. Since \LM T_{n}F\to T(T_{n}F)=T(F)\RM is a cubic extension over~$0$, the extension $U$ is over $X$.
\end{proof}

\begin{proposition}
The direction of $U$ is isomorphic to \LM \LDrvd{T}{n}{X}\Rm. Therefore, the congruence class represented by $U$ is a universal central $n$-extension of~$X$.
\end{proposition}
\begin{proof}
Recall from~\cite{EGVdL} that the centralisation $T_{n}F$ of $F$ is induced by dividing a certain object $T_{n}[F]$ out of the object~$F(n)$. (If $\R=\Ab(\X)$, then $T_{n}[F]$ is the join of the commutators in Definition~\ref{def:CentralExtension}, but in general this object is larger; see~\cite{Everaert-Gran-TT} for an explicit description in the general case.) As a consequence, the direction of~$T_{n}F$ is $\D(F)/T_{n}[F]$, the quotient of the direction of $F$ by this same object $T_{n}[F]$. The restriction of $T_{n}F$ to the commutator $[T_{n}F,T_{n}F]_{\R}$ restricts the direction to
\[
\frac{\D(F)\cap [F(n),F(n)]_{\R}}{T_{n}[F]},
\]
which is precisely the Hopf formula for $\L_{n}T(X)$; see Theorem~8.1 in~\cite{EGVdL}.

Concerning the second part of the statement, the needed inverse to
\[
\upsilon^{U}\from \HomFnctr{\R}{ \L_{n}T(X) }{ - }\To \Ext^{n}(X,-)|_{\R}\from \R \longrightarrow \AbGrps
\]
is determined by the weak universality of $U$.
\end{proof}

Thus we proved:

\begin{theorem}
Suppose $\R$ is an abelian subvariety of semiabelian variety $\X$ satisfying the \emph{Smith is Huq} condition~\ref{SH}. Given an $(n-1)$-acyclic object $X$, its universal $n$-fold central extension $[U]$ may be constructed as follows:
\begin{enumerate}
\item consider an $n$-fold presentation $F$ of $X$;
\item centralise $F$ to obtain the weakly initial $n$-cubic central extension $T_{n}F$;
\item take commutators pointwise to obtain the (weakly initial) $n$-cubic central extension \LM [T_{n}F,T_{n}F]_{\R}\Rm;
\item take kernels to obtain a \LM \langle 3\rangle ^n\Rm-diagram $U$, which represents an $n$-fold universal central cohomology class $[U]$.
\end{enumerate}
\end{theorem}

\subsection{Comparison With the Standard Definition}\label{Classical}
Universal central (one-fold) extensions are usually defined as initial objects amongst all central extensions of a given (necessarily perfect) object---see~\cite{CVdL, Gran-VdL} for an account in the context of semi-abelian categories. Let us explain why universality as defined above coincides with this standard concept. We assume that $X$ is a perfect object in a semi-abelian category with enough projectives satisfying the condition (UCE)~\cite{CVdL, GrayVdL1} and $\R=\Ab(\X)$. Then indeed, given a central extension
\[
\xymatrix{0 \ar[r] & H \ar[r]^{h} & U \ar[r]^{u} & X \ar[r] & 0}
\]
both $\langle \eta_{U},u\rangle$ and $\langle0,u\rangle\colon U\to \ab(U)\times X$ induce maps $H\to \ab(U)$. Uniqueness of the induced map implies $\eta_{U} h=0$. Now $\eta_{U} h\colon H\to \ab(U)$ is a regular epimorphism by the long exact homology sequence and the fact that $X$ is perfect, so $\ab(U)=0$: that is, also $U$ is perfect as in~\cite[Lemma~5.5]{Milnor}. Perfectness of the middle object in a weakly universal central extension implies universality~\cite[Proposition~5.3]{CVdL}, so the result follows under the condition (UCE). As explained in~\cite{CVdL}, unlike what happens for groups~\cite{Milnor}, $2$-acyclicity of the middle object $U$ need not be sufficient for universality in general semiabelian categories.

\subsection{The Standard Definition Does Not Extend to Higher Degrees}
This standard approach to universal central extensions as initial objects does not naively extend to higher degrees. A quick argument goes as follows. Consider a double cubic central extension
\[
\xymatrix{X \ar[d] \ar[r] & C \ar[d]\\
D \ar[r] & P}
\]
 over some object $P$, and assume it is universal in the strict sense; i.e. it is an initial object amongst double cubic central extensions over $P$. Taking into account that any double cubic extension can be centralised, and that any two regular epimorphisms over $P$ give rise to some double cubic extension, we see that, say, $f\colon C\to P$ factors \emph{uniquely} over any given regular epimorphism. As a consequence, $C$, and hence also $P$, is trivial. To see this, it suffices to note that $\langle g,f\rangle\colon C\to A\times P$ can only be unique for any given $A$ if $C$ is an initial object, so $C=0$.

\appendix

\section{Homological Considerations}
\label{sec:HomologyBackground}%

The following diagrammatic view of chain complexes and homology is intended to reveal the completely elementary nature of the proofs of~\ref{Lemma-Yoneda-Commutes-Homology} and~\ref{Theorem-Yoneda-Connecting-Map}. As in Mac\,Lane~\cite[p.~259]{MacLane:Homology}, let \LM \ChnCat\RM be the category depicted thus:
$$
\xymatrix@R=5ex@C=3em{
\cdots \ar[r] &
 n+1 \ar[r]^-{d_{n+1}} \ar[rd]_{t_{n+1}} &
 n \ar[r]^-{d_n} \ar@<+2pt>[d]^{t_n} &
 n-1 \ar[r] &
 \cdots \\
&& z \ar@<+2pt>[u]^{i_n} \ar[ru]_{i_{n-1}}
}
$$
The object set is \LM \CPrdct{\ZNr}{\{ z\}}\Rm, and the morphisms are generated by the ones displayed (for each \LM n\in \ZNr\Rm), subject to the relations $t_ni_n = \IdMap{z}$ and $d_nd_{n+1} = i_{n-1}t_{n+1}$. The category of chain complexes in an abelian category $\A$ is isomorphic to the category of functors \LM \ChnCat \to \A\RM sending $z$ to $0$. If \LM C\from (\ChnCat,z)\to (\A,0)\RM is a chain complex, we write $d^{C}_{n}\from C_n \to C_{n-1}$ for $C(d_n)\from C(n) \to C(n-1)$. The category \LM \ChnCat\RM is self dual via the isomorphism determined by $\omega\from \ChnCat \to \ChnCat^{\op}$, $\omega(n) \DefEq -n$, $\omega(z)\DefEq z$.

\subsection{Homology}\label{subsec:Homology}%
Given a chain complex $C$ in $\A$, its homology in position $n$ is usually introduced via the cokernel construction below:
$$
\xymatrix@R=5ex@C=3em{
& \Ker{d^{C}_{n}} \ar@{-{ >>}}[r] \ar@{{ |>}->}[d] &
 \CoKer{\widetilde{d}^{C}_{n+1}} \EqDef \H_nC \EqDef \HmlgyCoKer{n}{C} \\
C_{n+1} \ar[r]_-{d^{C}_{n+1}} \ar[ru]^{\widetilde{d}^{C}_{n+1}} &
 C_n \ar[r]_-{d^{C}_{n}} &
 C_{n-1}
}
$$
In addition, we obtain a chain complex in \LM \A^{\op}\RM via the composite \LM D\DefEq C^{\op}\omega\Rm:
$$
\xymatrix@R=5ex@C=3em{
\ChnCat \ar[r]^-{\omega} &
 \ChnCat^{\op} \ar[r]^-{C^{\op}} &
 \A^{\op}
}
$$
Then \LM \HmlgyCoKer{-n}{D}\RM is given by the diagram on the left below, which resolves to the diagram on the right.
\[
\resizebox{\textwidth}{!}{
\xymatrix@R=5ex@C=2em{
& \Ker{(d^{C}_{n+1})^{\op}} \ar@{-{ >>}}[r] \ar@{{ |>}->}[d] &
 \HmlgyCoKer{n}{C^{\op}} = \HmlgyCoKer{-n}{D} \\
C_{n-1} \ar[r]_-{(d^{C}_{n})^{\op}} \ar[ru]^{\widetilde{({d}^{C}_{n})^{\op}}} &
 C_n \ar[r]_-{(d^{C}_{n+1})^{\op}} &
 C_{n+1}
}\quad
\xymatrix@R=5ex@C=2em{
& \CoKer{d^{C}_{n+1}} \ar[ld]_{\overline{d}^{C}_{n}} &
 \HmlgyCoKer{n}{C^{\op}} \EqDef \HmlgyKer{n}{C} \ar@{{ |>}->}[l] \\
C_{n-1} &
 C_n \ar[l]^-{d^{C}_{n}} \ar@{-{ >>}}[u] &
 C_{n+1} \ar[l]^-{d^{C}_{n+1}}
}}
\]
Thus the cokernel construction of the homology on \LM C^{\op}\RM is a kernel construction of a homology object of $C$. Via the snake lemma applied to the diagram below, we see that these two constructions are naturally isomorphic.
\[
{\xymatrix@R=5ex@C=3em{
C_{n+1} \ar[r]^{d^{C}_{n+1}} \ar[d]_{\widetilde{d}^{C}_{n+1}} &
 C_n \ar@{=}[d] \ar@{-{ >>}}[r] &
 \CoKer{d^{C}_{n+1}} \ar[d]^{\overline{d}^{C}_{n}} \\
\Ker{d^{C}_n} \ar@{{ |>}->}[r] &
 C_n \ar[r]_{d^{C}_n} &
 C_{n-1}
}}
\]
Now consider a short exact sequence of chain complexes:
$$
\xymatrix@R=5ex@C=1.5em{
0 \ar[r] &
 A \ar@{{ |>}->}[rr] &&
 B \ar@{-{ >>}}[rr] &&
 C \ar[r] &
0
}
$$
There is an associated long exact sequence of kernel constructed homology objects. We explain the construction of the connecting morphism \LM \partial^{\mathrm{k}}_{n}\from \HmlgyKer{n+1}{C} \to \HmlgyKer{n}{B}\RM via the diagram below.
\[
\xymatrix@R=5ex@C=2em{
& A_{n+2} \ar@{{ |>}->}[r] \ar[d] &
 B_{n+2} \ar@{-{ >>}}[r] \ar[d] &
 C_{n+2} \ar[d]_{d^{C}_{n+2}} &
 \HmlgyKer{n+1}{C} =\Ker{\overline{d}^{C}_{n+1}} \ar@{{ |>}->}[d] \\
\HmlgyKer{n}{A} \ar@{{ |>}->}[d] &
 A_{n+1} \ar@{{ |>}->}[r] \ar[d]^{d^{A}_{n+1}} &
 B_{n+1} \ar@{-{ >>}}[r] \ar[d] &
 C_{n+1} \ar@{-{ >>}}[r] \ar[d]_{d^{C}_{n+1}} &
 \CoKer{d^{C}_{n+2}} \ar[ld]^{\overline{d}^{C}_{n+1}} \\
\CoKer{d^{A}_{n+1}} \ar[rd]_{\overline{d}^{A}_{n}} &
 A_n \ar[d]^{d^{A}_{n}} \ar@{{ |>}->}[r] \ar[d] \ar@{-{ >>}}[l] &
 B_n \ar@{-{ >>}}[r] \ar[d] &
 C_n & \\
& A_{n-1} \ar@{{ |>}->}[r] &
 B_{n-1}
}
\]
\begin{enumerate}[(1)]
\item Form the pullback
$$
\xymatrix@R=5ex@C=3em{
X_{n+1} \pullback \ar@{-{ >>}}[rr] \ar[d] &&
 \HmlgyKer{n+1}{C} \ar@{{ |>}->}[d] \\
B_{n+1} \ar@{-{ >>}}[r] &
 C_{n+1} \ar@{-{ >>}}[r] &
 \CoKer{d^{C}_{n+2} }
}
$$
\item The composite \LM X_{n+1} \to B_{n+1} \to B_n\to C_{n}\RM is $0$. So \LM {X_{n+1} \to B_{n+1} \to B_n}\RM lifts uniquely to an arrow \LM X_{n+1}\to A_{n}\Rm, and composes to \LM {X_{n+1}\to \CoKer{d^{A}_{n+1} }}\Rm.
\item The composite \LM X_{n+1}\to \CoKer{d^{A}_{n+1}}\to A_{n-1}\RM vanishes. So the morphism \LM X_{n+1}\to \CoKer{d^{A}_{n+1}}\RM lifts uniquely to a morphism \LM \overline{\partial}^{\mathrm{k}}_{n+1}\from X_{n+1}\to \HmlgyKer{n}{A}\Rm.
\item Let \LM K\DefEq \Ker{X_{n+1}\to \HmlgyKer{n+1}{C}} = \Ker{ B_{n+1}\to \CoKer{d^{C}_{n+2}} }\Rm. We claim that
the composite \LM K \to X_{n+1} \to \HmlgyKer{n}{A}\RM vanishes. Indeed, the canonical morphisms \LM A_{n+1}\to X_{n+1}\RM and \LM B_{n+2}\to X_{n+1}\RM lift to $K$ and conspire to form an epimorphism \LM A_{n+1}\oplus B_{n+2} \to K\Rm. Composing this arrow with \LM \overline{\partial}^{\mathrm{k}}_{n+1}\RM gives zero. Hence \Lm \overline{\partial}^{\mathrm{k}}_{n+1}\RM factors uniquely through \LM \HmlgyKer{n+1}{C}\RM to give \LM \partial^{\mathrm{k}}_{n+1}\from {\HmlgyKer{n+1}{C}\to \HmlgyKer{n}{A}}\Rm.
\end{enumerate}

Essential in the construction of \LM \partial^{\mathrm{k}}_{n+1}\RM is that it only involves a certain pullback, combined with unique lifting/factorisation situations. Therefore dualisation turns the chain complexes into cochain complexes, the kernel construction of homology objects into the cokernel construction of cohomology objects, and the connecting morphism \LM \delta_{\mathrm{c}}^{n+1}\from \CoHmlgyCoKer{n}{A^{\op}}\to \CoHmlgyCoKer{n+1}{C^{\op}}\RM now relies on a certain pushout, combined with unique factorisation/lifting situations. In the proof of Theorem~\ref{Theorem-Yoneda-Connecting-Map}, cochain complexes of Hom-functors assume the role of the dual complexes of $A$, $B$, and~$C$. Applying now \LM \NatTrafos{\cdot}{T}\RM just reverses the dualisation procedure by which we obtained \LM \delta_{\mathrm{c}}^{n+1}\Rm, and yields the connecting morphism \LM \partial^{\mathrm{k}}_{n+1}\RM of the short exact sequence of chain complexes \LM 0\to TA\to TB\to TC\to 0\Rm. This is the reason why the Yoneda-homology commutation is compatible with the connecting morphism.

The following lemma is needed in Section~\ref{sec:UCEs}.

\begin{lemma}
\label{lem:AbelianUCT}
In an abelian category $\A$ with enough projectives, fix $n\geq 1$ and consider a bounded-below positionwise projective chain complex
$$
\xymatrix@R=5ex@C=3em{
\cdots \to C_{n+1} \ar[r]^-{d_{n+1}} &
 C_n \ar[r]^-{d_n} &
 C_{n-1} \to \cdots \to C_0\to 0
}
$$
with \LM \H_iC=0\Rm, whenever \LM 0\leq i\leq n-1\Rm. Then for any $A$ in $\A$:
\begin{enumerate}
\item \Lm \H^n(C,A) \cong \HomFnctr{\A}{\H_nC }{ A }\Rm;
\item \Lm \H^i(C,A) = 0\RM for \LM 0\leq i\leq n-1\Rm.
\end{enumerate}
\end{lemma}
\begin{proof}
(i)\quad Consider the exact sequence below:
$$
\xymatrix@R=5ex@C=2em{
K\DefEq \Ker{d_{n-1}} \ar@{{ |>}->}[r] &
 C_{n-1} \ar[r] &
 \cdots \ar[r] &
 C_0 \ar@{-{ >>}}[r] &
 0
}
$$
The key is that, omitting $K$, we are left with a segment of a projective resolution of the $0$-object, and so $K$ is ``syzygied'' to~$0$. This implies the claim with an elementary computation. %
 --- Part (ii) follows because, for  \LM 0\leq i\leq n-1\Rm,
\[
\H^i(C,A) \cong \Ext^i(0,A)\cong 0.\qedhere
\]
\end{proof}

\subsection*{Added in proof}
At the 2015 Categorical Algebra workshop in Gargnano, Italy, George Janelidze informed us of his perception that the homological Yoneda isomorphism should really be seen as the  \LM \AbGrps\Rm-enriched Yoneda isomorphism plus `something else' about the category of functors  \LM \A\to \AbGrps\Rm. This perception turned out to be fruitful, as it resulted in the following new insights:
\begin{enumerate}[(i)]
\item Right exact functors are injective relative to the kernels of maps between representable functors  \LM \HomFnctr{\A}{ X }{ - }\Rm.
\item In the setting of Lemma \ref{Lemma-Yoneda-Commutes-Homology}, the homological Yoneda-isomorphism follows then with this computation:
$$
\H_nTC \cong \H_n\NatTrafos{ \HomFnctr{\A}{ C }{ - } }{T} \cong \NatTrafos{ \H^n \HomFnctr{\A}{ C }{ - } }{T}\ .
$$
\end{enumerate}
We observe further \cite[7.11]{Freyd} that globally injective functors are always right exact. However, \cite[p.~Chap. 7]{Freyd}, the converse is not true as there are right exact functors, and even exact ones, which are not globally injective.


\providecommand{\noopsort}[1]{}
\providecommand{\bysame}{\leavevmode\hbox to3em{\hrulefill}\thinspace}
\providecommand{\MR}{\relax\ifhmode\unskip\space\fi MR }
\providecommand{\MRhref}[2]{%
  \href{http://www.ams.org/mathscinet-getitem?mr=#1}{#2}
}
\providecommand{\href}[2]{#2}

\end{document}